\documentclass[a4paper,12pt,reqno]{article}
 \usepackage[english]{babel}
  \usepackage{amsmath,amsfonts,amssymb,amsthm,bbm}
   \usepackage{graphics,epsfig,psfrag} %%add this and next lines if pictures
   \usepackage{paralist,subfigure}
   \usepackage[dvipsnames]{xcolor}
   \usepackage{hyperref} %%Warning: when you first run your tex
    \usepackage[active]{srcltx} %%allows you to jump from your editor
    \usepackage{color,soul} %%allows you to use the \st{...} command, which
    \usepackage[utf8]{inputenc}
\usepackage[T1]{fontenc}
    \usepackage{authblk}
\newtheorem{theorem}{Theorem}[section]

\newtheorem{lemma}[theorem]{Lemma}
\newtheorem{prop}[theorem]{Proposition}

\theoremstyle{definition}
\newtheorem{definition}[theorem]{Definition}
\newtheorem{req}{Remark}

\newtheorem{question}{Question}

\DeclareMathOperator\inter{int}

\def\N{\mathbb{N}}
\def\Z{\mathbb{Z}}

\def\R{\mathbb{R}}
\def\C{\mathbb{C}}

\let\e=\varepsilon

\let\t=\tilde
\let\ol=\overline
\let\ul=\underline
\let\.=\cdot
\let\0=\emptyset

\let\mc=\mathcal

\def\1{\mathbbm{1}}

\def\l{\lambda_1}

\newenvironment{formula}[1]{\begin{equation}\label{#1}}
                       {\end{equation}\noindent}

\def\Fi#1{\begin{formula}{#1}}
\def\Ff{\end{formula}\noindent}

\title{\bf  Threshold phenomenon and traveling waves for heterogeneous integral equations and epidemic models.}

\author[]{Romain {\sc Ducasse}}
\affil[]{Universit\'e de Paris and Sorbonne Universit\'e, CNRS, Laboratoire Jacques-Louis Lions (LJLL), F-75006 Paris, France}	
\begin{document}

\date{}
\maketitle

%-------------------------------------------------------------------------------



\begin{abstract}
We study some anisotropic heterogeneous nonlinear integral equations arising in epidemiology. We focus on the case where the heterogeneities are spatially periodic. In the first part of the paper, we show that the equations we consider exhibit a \emph{threshold phenomenon}. In the second part, we study the existence and non-existence of \emph{traveling waves}, and we provide a formula for the admissible speeds. In a third part, we apply our results to a spatial heterogeneous SIR model.
\end{abstract}

\section{Introduction}

\subsection{Motivations: spatial models for the spread of epidemics}\label{sec intro}

%The mathematical modeling of infectious diseases aims at understanding the spread and outcome of epidemics. To do so, a fruitful approach consists in building models from simple assumptions on the mechanisms of the epidemic, and to investigate their qualitative properties.\\

In $1927$, Kermack and McKendrick introduced in \cite{KmcK1, KmcK2, KmcK3} several deterministic models describing the evolution of a disease in a closed population. Their most general model consists in a \emph{renewal equation for the infection}, and takes the form of the following nonlinear Volterra integral equation:
\begin{equation}\label{ren ode}
u(t) = \int_0^{t} \Gamma(\tau)g(u(t-\tau))d\tau + f(t), \quad t>0.
\end{equation}
The unknown function $u(t)\geq 0$ (the \emph{cumulative force of infection}) represents ``how much'' the population is contaminated by time $t>0$; $u=0$ means no contamination, while $u>0$ means that a proportion of the population is infected. The kernel $\Gamma(\tau)$ encodes the characteristics of the epidemic (mean duration of the contamination, incubation period...). The function $g$ reflects the nonlinear growth of the epidemic. Finally, the function $f$ accounts for the initial infectivity. We refer to \cite{formulation} for more details concerning the modelling aspects.\\

Kermack and McKendrick also introduced  in the mentioned papers the {\em SIR model}. It consists in a set of coupled ODEs, and became of great importance in mathematical epidemiology, to such an extend that it sometimes overshadows the more general model \eqref{ren ode}, which encompasses not only the SIR model, but also many more. \\

One important limitation of these models is that they do not take into account \emph{spatial effects}, such as diffusion and migration of individuals. Such effects are now recognised as being of first importance in the understanding of propagation of epidemics. To bridge this gap, Diekmann \cite{Di1} and Thieme~\cite{T1} introduced independently in $1977$ the following spatial generalization of~\eqref{ren ode}:
\begin{equation}\label{eq}
u(t,x) = \int_0^{t} \int_{y\in \R^N} \Gamma(\tau,x,y)g(u(t-\tau,y))dyd\tau + f(t,x), \quad t>0,\ x\in \R^N.
\end{equation}
As in \eqref{ren ode}, the unknown is the function $u$, which still represents the strength of infection, now depending not only on the time variable $t$, but also on a space variable~$x$. We refer to the original papers  \cite{Di1, T1} for the details concerning the modelling.\\

While equation \eqref{eq} describes the evolution of an epidemic in a population where some infected individuals are introduced at a given initial time,
%Similarly to the original Kermack-McKendrick model, this generalization encompasses (when specifying the coefficients) some \emph{spatial SIR models} (see Section \ref{sec SIR}). \\
%
it is also interesting to study the propagation of an epidemic \emph{without assuming any specific initial condition}, to see the ``generic'' way the epidemic spreads through space. In this case, it is natural to consider the same problem but with solutions defined for all time $t\in \R$; this allows to find \emph{traveling waves} solutions. This is possible by considering the second form of the model by Diekmann and Thieme:
\begin{equation}\label{eq waves}
u(t,x) = \int_0^{+\infty} \int_{y\in\R^N} \Gamma(\tau,x,y)g(u(t-\tau,y))dyd\tau, \quad t\in \R,\ x\in \R^N.
\end{equation}
Diekmann and Thieme \cite{Di1, Di2, T1, T2, T3, TZ} studied equations \eqref{eq}, \eqref{eq waves} in the homogeneous, isotropic case, that is, under the assumption that
\begin{equation}\label{hyp iso}
\Gamma(t,x,y) = \Lambda(t,\vert x-y \vert),
\end{equation}
for some $\Lambda : \R^2 \to \R$, i.e, $\Gamma$ only depends on the distance between the points (and on $t$). See also \cite{Ar} for related results.\\
%for instance  Diekmann \cite{Di1, Di2}, Thieme \cite{T1, T2, T3}, Thieme and Zhao \cite{TZ}, Aronson \cite{Ar}. \\
%In the case where the domain is not the whole space $\R^N$ but a bounded domain, some results were obtained by Inaba \cite{I}.\\
%

However, hypothesis \eqref{hyp iso} is very retrictive from the point of view of modelling: it means that that everything in the model - the medium, the initial population, the recovery and contamination rates - is spatially homogeneous (this appears clearly when considering SIR models, see the discussion in Section \ref{sec SIR} below - we refer to the original papers \cite{Di1, T1} for more general considerations).\\

This paper is dedicated to the study of equations \eqref{eq}, \eqref{eq waves} without this isotropy condition - we focus on the more general case of periodic heterogeneous media, that is, situations where $\Gamma$ satisfies (without loss of generality, we consider only the $1$-periodic case throughout the whole paper)
\begin{equation}\label{hyp per}
\Gamma(t, x +k, y+ k) = \Gamma(t, x, y), \quad \forall t \in \R,\ \forall x,y \in \R^N,\ \forall k \in \Z^N.
\end{equation}
We consider here two aspects of equations \eqref{eq}, \eqref{eq waves}: first, we study under which conditions initial disturbances will propagate (the {\em threshold effect}), second, we study the existence of {\em traveling waves}. Therefore, this paper can be seen as a generalization of the papers of Diekmann and Thieme mentionned above to an heterogeneous setting.\\

%Without loss of generality, we consider only the $1$-periodic case throughout the whole paper, without further notice. Other technical hypotheses are given later.\\
The techniques we will employ come from the study of heterogeneous KPP (Kolmogorov-Petrovski-Piskunov \cite{KPP}) reaction-diffusion equations. Those are PDEs of the form $\partial_t u = \nabla (A(x)\nabla u) + f(u)$, with $f$ satisfying some concavity assumption. Indeed, if $u(t,x)$ is solution to such a reaction-diffusion equation, then one can see that (under suitable regularity assumptions), $u(t,x)$ also solves an equation of the form \eqref{eq}, with $\Gamma(t,x,y)$ being the fundamental solution of a parabolic operator. Many methods were introduced in the last years to study heterogeneous reaction-diffusion equations (see \cite{BReigen} for instance), and as we will see, they can somewhat be adapted to the setting of nonlinear integral equations.\\

%Let us emphasize that the interest of equations \eqref{eq} and \eqref{eq waves} goes beyond scalar reaction-diffusion equations and epidemiological models. Up to choosing a different set of $\Gamma,f,g$, they rewrite as reaction-diffusion systems, as delayed and non-local models from population dynamic (see \cite{TZ}), or as the neural field equation from neurosciences (see \cite{Cowan}).\nota{virer ?}\\

One motivation in studying \eqref{eq}, \eqref{eq waves} is to obtain results for SIR models. Indeed, just like the original model of Kermack and McKendrick \eqref{ren ode} encompasses the standard SIR model (recalled below in section \ref{sec SIR}), equations \eqref{eq}, \eqref{eq waves} encompass some spatial SIR models. As a by-product of our analysis of these integral equations, we will obtain new results for some heterogeneous SIR model. Before presenting our main results, we recall in the next section some facts about SIR models and we explain why they are special cases of the renewal equations presented here. This discussion is also enlightening to understand what $\Gamma,g,f$ represent in \eqref{eq},~\eqref{eq waves}.

\subsection{Connection with SIR models}\label{sec SIR}

SIR systems are compartmental models, that is, the population is divided into several classes (the compartments), that interact following some simple rules. The first SIR model was introduced by Kermack and McKendrick in their paper \cite{KmcK1}, as a special case of their general model \eqref{ren ode}. It takes the form of the following set of ODE:
\begin{equation}\label{ode sir}
\left\{
\begin{array}{rll}
\dot S(t) &= - \alpha S(t) I(t),& t \in \R,\\
\dot I(t) &=  \alpha S(t) I(t) - \mu I(t),& t \in \R,\\
\dot R(t) &=  \mu I(t),& t \in \R,\\
\end{array}
\right.
\end{equation}
where $\alpha,\mu>0$. The functions $S,I,R$ are the unknowns, and represent respectively the number of {\em Susceptible, Infectious} and {\em Recovered} individuals in the population. The infectious individuals contaminate the susceptible ones, following a law of mass-action, that is, the rate of contamination is $\alpha I$. The infectious individuals recover with rate $\mu$. Observe that the function $R$, the recovered, does not play any role in the dynamics of the system. For this reason, we will not mention it in the sequel.\\

%The SIR model \eqref{ode sir} is a specific case of the renewal equation \eqref{ren ode}. Indeed, if $S(t),I(t),R(t)$ solves \eqref{ode sir}, with initial datum $(S_0,I_0,R_0)$\nota{dire un mot sur $R$}, then 
%%$u(t) = -\ln\left(\frac{S(t)}{S_0}\right)$ 
%$u(t) = -\ln(S(t)/S_0)$ solves the renewal equation \eqref{ren ode} with $g(z)=1-e^{-z}$ and $\Gamma(\tau) = \alpha S_0 e^{-\mu}$, and $f(t) = \frac{\alpha}{\mu} I_0$.\nota{check, pas bon pour f}\\

 Many spatial generalizations of the SIR model \eqref{ode sir} were introduced. We pay in this paper a particular attention to the following one
 \begin{small}
   \begin{equation}\label{SIR 2}
\left\{
\begin{array}{rll}
\partial_t S(t,x) &= - \alpha(x)S(t,x) \int_{y\in \R^N}K(x,y)I(t,y)dy,& t >0 ,\ x\in \R^N,\\
\partial_t I(t,x) &=  \alpha(x)S(t,x) \int_{y\in \R^N}K(x,y)I(t,y)dy - \mu(x) I(t,x),& t >0,\ x\in \R^N.
%\partial_t R(t,x) &= \mu(x) I(t,x),&\quad t>0, \ x\in \R^N,
\end{array}
\right.
\end{equation}
 \end{small}
%  \begin{equation}\label{SIR 2}
%\left\{
%\begin{array}{rll}
%\partial_t S(t,x) &= - \alpha(x)S(t,x) \int_{y\in %\R^N}K(x,y)I(t,y)dy,& t \in \mc I,\ x\in \R^N,\\
%\partial_t I(t,x) &=  \alpha(x)S(t,x) \int_{y\in %\R^N}K(x,y)I(t,y)dy - \mu(x) I(t,x),& t \in \mc I,\ x\in %\R^N.
%%\partial_t R(t,x) &= \mu(x) I(t,x),&\quad t>0, \ x\in \R^N,
%\end{array}
%\right.
%\end{equation}
The functions $S(t,x),I(t,x)$ represent the densities of susceptible and infected indivuals respectively. The contamination is non-local, that is, the susceptible individuals located at point $x$ can get contaminated by infectious located at an other point $y$, with some probability $K(x,y)$. The rate of infection at point $x$ at time $t$ is $\alpha(x)\int_{y\in \R^N}K(x,y)I(t,y)dy$. The recovery rate is a function $\mu(x)$. The fact that it can vary from places to places may account for the effects of localized quarantine zones or different vaccination policies, for instance.\\

The connection between the SIR models and the renewal equations is known since the pionering works of Kermack and McKendrick: up to doing some change of functions - sometimes called the \emph{linear chain trick}, see \cite{formulation} - one can turn SIR models into renewal equations.

Indeed, if $(S(t,x),I(t,x))$ solves \eqref{SIR 2} with initial datum $(S_0(x),I_0(x))$\footnote{We say that the couple $(S,I)$ solves \eqref{SIR 2} if it is $C^1$ in $t>0$, $C^0$ on $[0,+\infty)\times \R^N$, if \eqref{SIR 2} is satisfied pointwise for $(t,x) \in (0,+\infty)\times \R^N$, and if $(S(t,\cdot),I(t,\cdot))\to (S_0,I_0)$ as $t$ goes to $0$ pointwise in $x$. }, where $S_0 >0$, then $u(t,x) := -\ln \left(\frac{S(t,x)}{S_0(x)}\right)$ solves \eqref{eq} with $g(z) = 1-e^{-z}$ and
\begin{multline}\label{gamma 2}
\Gamma(t,x,y) = \alpha(x)S_0(y)e^{-\mu(y) t} K(x,y) ,\\  \quad f(t,x) = \alpha(x)\int_0^t\int_{y\in \R^N}K(x,y)I_0(y) e^{-\mu(y)\tau}d\tau dy.
%\int_{0}^{t}\int_{y\in \R^N}\Gamma(\tau,x,y)\frac{I_0(y)}{S_0(y)} dy d\tau.
%\quad f(t,x) = \alpha(x)\int_{0}^{t}\int_{y\in \R^N}\mc H(\tau,x,y)I_0(y) dy d\tau,
\end{multline}
The details are given in Section \ref{sec comp}. As annouced before, we see that $\Gamma$ indeed encodes the caracteristics of the epidemic (contamination rate, recovery rate) and of the initial susceptible population, while $f$ accounts for the initial infection.

If we wanted to study system \eqref{SIR 2} for $t\in \R$ (rather than for $t>0$), without specifying any initial datum (which will be the case in order to find traveling waves), then, we can turn \eqref{SIR 2} into \eqref{eq waves} with similar $\Gamma,g$. Again, see Section \ref{sec comp} for details.\\

 The system \eqref{SIR 2} was originally considered by Kendall in \cite{K1, K2}, under the assumption that everything is homogeneous, that is, $\alpha, \mu$ are positive constants, the contamination kernel $K(x,y)$ is a decreasing function of the distance only (i.e., $K(x,y) = K(\vert x - y\vert)$), and the initial datum for $S$ is constant.

 In this case, the function $\Gamma(t,x,y)$ given by \eqref{gamma 2} satisfies the isotropy hypothesis \eqref{hyp iso}, and the model of Kendall can be studied using the results on the homogeneous renewal equation mentioned above (see for instance \cite{M} and references therein).
 
 As soon as something in the model \eqref{SIR 2} is not homogeneous, this is not possible. By studying \eqref{eq}, \eqref{eq waves} under the hypothesis \eqref{hyp iso}, we will be able to consider some spatial heterogeneous SIR models.

%Finally, if $(S(t,x),I(t,x))$ is a solution of \eqref{SIR 2} with $\mc I = \R$ such that $S(t,x)\to S_{-\infty}(x)$, for some $S_{-\infty}>0$, and such that $I(t,x)\to 0$, as $t$ goes to $-\infty$, locally uniformly in $x \in \R^N$ we can prove that the function $u(t,x) := \left(\frac{S(t,x)}{S_{-\infty}(x)}\right)$ solves \eqref{eq waves} with the same $\Gamma,g$ as in \eqref{gamma 2} but with $S_0$ replaced by $S_{-\infty}$. Again, we refer to the Appendix~\ref{app SIR}\nota{non, ce sera une section plutôt ?} where we provide the computations.\\

 %The SIR model \eqref{ode sir} was first generalized to the spatial case by D. G. Kendall \cite{K1}, who introduced the following set of integro-differential equations
 % \begin{equation}\label{SIR K}
%\left\{
%\begin{array}{rll}
%\partial_t S(t,x) &= - \alpha  S(t,x) \int_{y\in \R^N}K(\vert x - y\vert )I(t,y)dy,& t >0,\ x\in \R^N,\\
%\partial_t I(t,x) &=  \alpha S(t,x) \int_{y\in \R^N}K(\vert x - y \vert )I(t,y)dy - \mu I(t,x),& t >0,\ x\in \R^N.
%%\partial_t R(t,x) &= \mu(x) I(t,x),&\quad t>0, \ x\in \R^N,
%\end{array}
%\right.
%\end{equation}

 \begin{req}
 Many other spatial SIR models were introduced. An interesting setting is to consider situations where the individuals can ``move''. This can be done by adding Laplace (or more general diffusion) operators in the equations for $S$ and $I$.

 Hosono and Ilyas \cite{HI} proved the existence of traveling waves for such SIR models in the homogeneous case. Ducrot and Giletti \cite{DG} proved the existence of waves, their stability and the existence of a threshold phenomenon in the heterogeneous periodic framework when only the infected diffuse (not the susceptibles) and with local contamination.
 %In this specific case, it is possible to turn the SIR system into a reaction-diffusion equation, using some change of function (different from the one used here) and to use the full strength of reaction-diffusion theory. Let us mention that our renewal equation approach presented here could also be used to consider this case, but it adds many technicalities, for this reason we only focus on the SIR model \eqref{SIR 2} without diffusion but with nonlocal contamination.
When both the susceptible and infectious individuals diffuse, much less is known (even in the homogeneous framework), and the renewal equation approach does not seem to apply anymore. The author considered the situation of a bounded domain in \cite{Duc}.

We also refer to \cite{BRRSIR} where diffusive SIR models with networks are considered with similar methods.
 \end{req}

\subsection{Propagation and generalized traveling waves}

%\subsubsection{Propagation and generalized traveling waves}

When studying models from epidemiology, the main questions that one may want to answer are the two following:
\begin{question}\label{q1}
Under what conditions does the epidemic propagate? Moreover, when the epidemic propagates, what is the final state of the population?
\end{question}

\begin{question}\label{q2}
How does the epidemic spreads through space? What is the ``speed'' of the epidemic?
\end{question}
Of course, these notions of propagation, of final state, of speed, must be adapted to the model under consideration. The next definition introduces the notions of \emph{propagation} and of \emph{fading out} for an epidemic described by the general model \eqref{eq}.

\begin{definition}[Propagation for \eqref{eq}]\label{def prop}
We say that the epidemic \emph{propagates} if the solution $u(t,x)$ of \eqref{eq} converges to some $u_{\infty}(x) \in L^{\infty}(\R^N)$ as $t$ goes to $+\infty$, locally uniformly in $x\in \R^N$, and if
\begin{equation*}\label{propagation}
\liminf_{\vert x \vert \to +\infty} u_{\infty}(x) >0.
\end{equation*}
%\begin{equation*}\label{propagation}
%\liminf_{\vert x \vert \to +\infty}\left(\liminf_{t\to+\infty} u(t,x)\right) >0.
%\end{equation*}
On the other hand, we say that the epidemic \emph{fades out} if the solution $u$ of \eqref{eq} converges similarly to some $u_{\infty} \in L^{\infty}(\R^N)$ such that
\begin{equation*}\label{no propagation}
\limsup _{\vert x \vert \to +\infty} u_{\infty}(x) = 0.
\end{equation*}
\end{definition}
In other words, the epidemic propagates if the infection eventually spreads everywhere.\\

To answer Question \ref{q1}, we will prove that \eqref{eq} exhibits a \emph{threshold phenomenon}, that is, we will identify a quantity $\lambda_1 \in \R$, depending on the characteristics of the epidemic and on the initial population, such that, if $\l$ is greater than some threshold, the epidemic propagates, no matter how ``small'' the initial infectivity. On the other hand, if $\l$ is below the threshold, then the epidemic fades out, no matter how ``large'' the initial infectivity.\\

The first proof that an epidemic model can exhibit a threshold phenomenon dates back to the paper of Kermack and McKendrick \cite{KmcK1} (we review some results in Section~\ref{sec results}).\\

To answer Question \ref{q2}, we will study the existence and non-existence of \emph{generalized traveling waves} for \eqref{eq waves}.
\begin{definition}[Traveling waves for \eqref{eq waves}]\label{def tw}
We say that a solution $u(t,x)$ of \eqref{eq waves} is a (generalized) traveling wave connecting $0$ to $U(x)\in L^{\infty}(\R^N)$ in the direction $e\in \mathbb{S}^{N-1}$ with speed $c > 0$ if
$$
 \sup_{x\cdot e - ct \leq \delta}\vert  u(t,x) - U(x)\vert\underset{\delta\to - \infty}{\longrightarrow} 0 \quad \text{ and } \quad  \sup_{x\cdot e - ct \geq \delta}\vert  u(t,x) \vert\underset{\delta\to + \infty}{\longrightarrow} 0.
$$
\end{definition}
%We will identify which states can be connected by a generalized traveling wave, and for which speeds there exist or not traveling waves.\\

The notion of generalized traveling waves was introduced, under a more general form, by Berestycki and Hamel in \cite{BHgtw} in the context of heterogeneous reaction-diffusion equations, in order to generalize the notion of traveling waves introduced by Kolmogorov, Petrovski and Piskunov \cite{KPP} for homogeneous reaction-diffusion equations. Let us mention that generalized traveling waves also generalize the concept of \emph{pulsating traveling waves}, see \cite{BHper, W} for more details.

The waves satisfying definition \ref{def tw} are sometimes called \emph{almost planar wave with linear speed}, see Definition 2.8 in \cite{BHgtw}.

%We mention that Weinberger \cite{W} introduced a general abstract theory to prove existence of (pulsating) traveling waves for monotone systems. It may be interesting to see if, and how, these results could apply to the situation we consider here. The main difficulty is to find the stability of the rest states. In the case considered here, we have an extra "KPP" structure that allows for a direct approach.
%
%
%
%

%
\begin{req}\label{req tw}
The waves we consider here have a linear speed. Observe that there are no \emph{a priori} reason for this to hold true. Actually, there are examples of reaction-diffusion equations where the propagation happens with a super-linear speed. For instance, Cabr\'e and Roquejoffre~\cite{CR} prove that this is the case for reaction-diffusion equations with diffusion given by a fractional Laplace operator $(-\Delta)^s$, $s\in(0,1)$. This comes from the fact that the transition function of the underlying process decays ``too slowly'' (algebraically) at infinity. Similar observation was made in the context of neural field equations, see for instance \cite{FF}. To prevent such super-linear propagation here, we will restrict our attention to kernels $\Gamma$ that decay exponentially fast. 
%If the kernel were to decay too slowly, it is natural to conjecture that there would be no waves, and that the spread would be super-linear.
\end{req}

%%%%%%%%%%%%%%%%%%%%%%%%%%%%%%%%%%%%%%%%%%%%%%%%%%%%%%%

%\subsubsection{Propagation and generalized traveling waves}

\subsection{Results of the paper}\label{sec results}

We gather in this section the main results of the paper. After stating some hypotheses, we present in Section \ref{sec res 1} our results concerning the general model \eqref{eq} and \eqref{eq waves}. Section \ref{sec res 2} contains an application of our results to the SIR model \eqref{SIR 2}.

\subsubsection{General hypotheses}\label{sec hyp}

 The hypotheses presented here are classical (see \cite{Di1, TZ}) and will be assumed throughout the whole paper, without further notice.\\

We assume that $g$ is a strictly increasing, bounded, Lipschitz continuous function on $[0,+\infty)$, such that $g(0)=0$, $g(z)> 0$ for all $z>0$. Moreover, we assume that
\begin{equation}\label{hyp g concave}
z \mapsto \frac{g(z)}{z} \ \text{ is strictly decreasing on}\ \R^+.
\end{equation}
In addition, $g(z)$ is differentiable at $z=0$, and there is $C>0$ such that,
\begin{equation}\label{hyp g}
g^{\prime}(0)z - C z^2 \leq g(z) \leq g^{\prime}(0)z, \quad  \text{for every}\ z\geq0.
\end{equation}
The right-hand side inequality is actually a consequence of \eqref{hyp g concave}.

We assume that $\Gamma \geq0$ is periodic in the sense \eqref{hyp per}, and that it is non-degenerate in the sense that there are $\e,  r >0$ such that $\Gamma(t,x,y) > \e$ if $  t   + \vert x - y\vert \leq r$. Moreover,
$$
\forall x\in \R^N, \quad \Gamma(\cdot, x, \cdot) \in L^1((0,+\infty)\times\R^N).
$$
We also assume the following regularity hypothesis: for every $\e>0$, there is $\delta>0$ such that, for every $x_1, x_2 \in \R^N$ such that $\vert x_1 - x_2 \vert \leq  \delta$, we have
\begin{equation}\label{hyp gamma 1}
\int_0^{+\infty}\int_{\R^N} \vert \Gamma(\tau, x_1, y) - \Gamma(\tau, x_2, y)\vert d\tau dy < \e.
\end{equation}
The function $f(t,x)$, that appears only in \eqref{eq}, not in \eqref{eq waves}, is supposed to be continuous on $[0,+\infty)\times \R^N$ and non-negative. Moreover, we assume that $f$ is non-decreasing with respect to the variable $t$, and that
\begin{equation}\label{hyp f}
f(t,x) \underset{t \to +\infty}{\nearrow} f_{\infty}(x), \ \text{ locally uniformly in } \ x\in \R^N,
\end{equation}
where $f_{\infty}$ is bounded, uniformly continuous on $\R^N$ and satisfies
\begin{equation}\label{hyp f 2}
f_{\infty}(x) \underset{\vert x \vert \to +\infty}{\longrightarrow} 0.
\end{equation}
Under these general hypotheses, Diekmann \cite{Di1} proved the existence, uniqueness and convergence of solutions to \eqref{eq}.
\begin{prop}[\cite{Di1}, Theorems 3.3 and 3.4]\label{prop cv}
There is a unique continuous bounded solution $u(t,x)$ to \eqref{eq}. Moreover, $u$ is time-nondecreasing and
$$
u(t,x) \underset{t\to +\infty}{\nearrow} u_\infty(x) \quad \text{locally uniformly in }\ x\in \R^N,
$$
where $u_\infty \in C^{0}(\R^N)$ is a solution of the \emph{limiting equation}:
\begin{equation}\label{eq limit}
u_{\infty}(x) = \int_{\R^N} V(x,y)g(u_{\infty}(y))dy + f_{\infty}(x), \quad x\in \R^N,
\end{equation}
where, for every $x$ and a.e. $y$,
\begin{equation}\label{def V}
V(x,y) := \int_{0}^{+\infty}\Gamma(\tau,x,y)d\tau.
\end{equation}
\end{prop}
%
%Let us emphasize again that this result does not require any structure hypothesis, neither the isotropic \eqref{hyp iso} nor the periodic one \eqref{hyp per}, to hold true. However, in the sequel, we will assume that the periodicity hypothesis \eqref{hyp per} is verified. It is readily seen that \eqref{hyp per}, combined with \eqref{hyp gamma 1}, implies \eqref{hyp gamma}, so that Proposition \ref{prop cv} holds true as soon as \eqref{hyp per} does.

\subsubsection{Results on the integral equations \eqref{eq} and \eqref{eq waves}}\label{sec res 1}

In addition to the general hypotheses presented above - that are assumed throughout the whole paper - we will need others hypotheses specific to the heterogeneous case considered here.

First, we will assume that $V$, given by \eqref{def V}, satisfies, for every compact set $S\subset \R^N$,
\begin{equation}\label{hyp gamma vp}
\sup_{x\in S}\left(\int_{y \in S} V^2(x,y)dy\right) < +\infty.
\end{equation}
We will also add a ``symmetry hypothesis'' on $V$ (\eqref{sym V} below) and a decay hypothesis on $\Gamma$ (\eqref{non exp} below).\\

The first part of the paper is concerned with the threshold phenomenon, that is, with the long-time behavior of solutions to \eqref{eq}. If $u(t,x)$ is such a solution, Proposition \ref{prop cv} tells us that it converges as $t$ goes to $+\infty$ to a function $u_\infty$ solution of \eqref{eq limit}. To establish whether the epidemic propagates or fades out in the sense of Definition \ref{def prop}, we have to look for the values of $u_\infty(x)$ for $\vert x \vert$ large. Hypothesis \eqref{hyp f 2} states that $f_\infty$ vanishes for large $\vert x \vert$, therefore, it is reasonable to infer that $u_\infty$ should be similar, at least for large $\vert x \vert$, to a solution~$U$ of
\begin{equation}\label{eq stat}
U(x) = \int_{\R^N}V(x,y)g(U(y))dy, \quad x\in \R^N.
\end{equation}
Clearly, the function $U\equiv 0$ is solution of \eqref{eq stat}. We will prove that the epidemic propagates if, and only if, there is a strictly positive solution to \eqref{eq stat}.

The key-point in our analysis is that the long-time behavior of \eqref{eq} is completely determined by the principal periodic eigenvalue of the linearization of \eqref{eq stat} near $U=0$, that is, the operator $L$, acting on the set of continuous $1$-periodic functions on $\R^N$, $C^{0}_{per}(\R^N)$:
\begin{equation}\label{L}
L\phi(x) := \int_{\R^N} g^{\prime}(0)V(x,y)\phi(y)dy.
\end{equation}
In order to make to make this work, we require a \emph{symmetry hypothesis} on the operator $L$: we assume that there are $\gamma_1, \gamma_2 \in C^{0}_{per}(\R^N)$, $\gamma_1, \gamma_2 >0$  and that there is $\t V(x,y)$ such that $\t V(x,y) = \t V(y,x)$ for every $x,y\in \R^N$, and
\begin{equation}\label{sym V}
V(x,y) = \t V(x,y)\gamma_1(x)\gamma_2(y).
\end{equation}
These hypothesis may seem surprising at first glance. However, it turns out that what we said above (the fact that the principal periodic eigenvalue of the operator gives the stability of the null state) does not hold true without it. We present a counterexample below as Proposition \ref{prop ce}, and we give more details there about this fact. Fortunately, as we will see, this hypothesis is automatically satisfied when working with SIR models.\\

We let $\l \in \R$ denote the \emph{principal periodic eigenvalue} of $L$, that is, the unique real number such that there is $\phi_p \in C^{0}_{per}(\R^N)$, $\phi_p >0$, such that
$$
L\phi_p = \l \phi_p.
$$
The existence of $\l$ is a consequence of the Krein-Rutman theorem, see \cite{KR}, that we recall in Section \ref{sec th} as Theorem \ref{th KR}.
Our first result states that $\l$ characterizes the number of solutions to \eqref{eq stat} and the threshold phenomenon.
\begin{theorem}\label{th threshold}
Assume that $f_\infty \not\equiv 0$. Let $u(t,x)$ be the solution to \eqref{eq} and let $u_\infty(x) := \lim_{t\to +\infty}u(t,x)$.
\begin{itemize}

\item If $\l >1$, the epidemic propagates in the sense of Definition \ref{def prop}. Moreover, there is a unique bounded positive solution $U$ to \eqref{eq stat}. It is periodic and
$$
\sup_{\vert x \vert \geq R}\vert u_{\infty}(x) - U(x)\vert \underset{R \to +\infty}{\longrightarrow} 0.
$$
\item If $\l \leq 1$, the epidemic fades out, in the sense of Definition \ref{def prop}. Moreover, there are no positive bounded solutions to \eqref{eq stat}, and
$$
\sup_{\vert x \vert \geq R}\vert u_{\infty}(x) \vert \underset{R \to +\infty}{\longrightarrow} 0.
$$
\end{itemize}

\end{theorem}
This completely answers Question \ref{q1} in the periodic case. In the homogeneous cas, this result was obtained by Diekmann and Thieme, see for instance \cite{T2}, Theorems 2.6a and 2.8c. In the homogeneous case, we have the following explicit formula $\l = g^{\prime}(0)\int_{\tau>0}\int_y \Lambda(\tau,y)d\tau dy$.\\

Our next result is dedicated to answering Question \ref{q2}, by studying the existence and non-existence of traveling waves to \eqref{eq waves}. Remembering Remark \ref{req tw} above, it is necessary, in order for such a result to hold true, to have some decay on $\Gamma, V$. Hence, we assume in the next results that, for every $\rho\geq 0$, there are $C, \e >0$ such that,
\begin{equation}\label{1 hyp waves per}
\Gamma(t,x,y) \leq C e^{-\e t}e^{-\rho \vert x - y \vert}.
\end{equation}
This hypothesis may be slightly relaxed, see Section \ref{sec tw}, but we will assume it when studying the existence of waves.

%there is $\rho_{0}>0$ such that, for every $\rho\in [0,\rho_0)$, 
%for every $\rho\geq 0$, for every $e\in \mathbb{S}^{N-1}$ and for every $c\geq 0$, the kernel\nota{mettre plutot que gamma fois exp est L1 puis mettre l'hyp optimale ensuite, enventuellement mettre un lemme?}
%\begin{equation}\label{hyp waves per}
%\Gamma_{\rho,c,e}(\tau,x,y) := \Gamma(\tau,x,y)e^{-\rho(c\tau + (y-x)\cdot e)}
%\end{equation}
%satisfies hypothesis \eqref{hyp gamma 1}, locally uniformly in $(\rho,c)$.\footnote{The next results actually hold true by assuming that hypothesis \eqref{hyp waves per} is verified only for $\rho \in[0,\rho_0)$. This would be important if we wanted to consider kernels $\Gamma$ that decay at most exponentially, which would be necessary if we wanted to apply our results to SIR models with diffusion in the $I$ individuals - however, we will not treat this case here.} 

%This hypothesis is the equivalent of hypothesis \eqref{hyp waves iso} in our periodic case. As mentioned in Remark \ref{req tw} above, this type of hypothesis prevents the propagation to be super-linear.

For $\rho , c\geq 0, e \in \mathbb{S}^{N-1}$, we define
$$
V_{\rho,c}(x,y) := \int_{0}^{+\infty}  \Gamma(\tau,x,y)e^{-\rho c \tau} d\tau,
$$
and the operator $\mc S_{\rho,c,e}$
\begin{equation}\label{Slc}
(\mc S_{\rho,c,e}\phi)(x) :=  \int_{\R^N}g^{\prime}(0)V_{\rho,c}(x,y)e^{-\rho(y- x) \cdot e} \phi(y)  dy.
\end{equation}
We let $\lambda_1(\rho,c,e)$ denote the principal periodic eigenvalue of the operator $\mc S_{\rho,c,e}$ acting on $C_{per}^{0}(\R^N)$. Observe that $\lambda_1(0,0,e)=\l$ defined above. We define
\begin{equation}\label{def c}
c^{\star}(e) := \inf\Big\{ c\geq 0 \ : \ \exists \rho \geq 0 \ \text{ such that }\ \l(\rho,c,e) \leq 1   \Big\}.
\end{equation}
%In the isotropic case, this boils down to the speed $c^\star$ defined in \eqref{speed iso}.

%Observe that, for every $\rho \in [0,\rho_0), e\in \mathbb{S}^{N-1}$, $\l(\rho,0,e) = \l$, where $\l$ is the principal periodic e%igenvalue of the operator $L$ defined by \eqref{L}.
%
\begin{theorem}\label{th waves}
%Assume that $\lambda_1\geq 1$. Then, there are no traveling waves solutions to \eqref{eq waves}.
Assume that $\Gamma$ satisfies \eqref{1 hyp waves per} and that $\l >1$. Let $U$ be the unique positive solution of~\eqref{eq stat} given by Theorem \ref{th threshold}. Then, for every $e\in \mathbb{S}^{N-1}$ and for every $c > c^\star(e)$, there is a generalized traveling wave solution to \eqref{eq waves}, connecting $0$ to $U$ in the direction $e$ with speed $c$. 
%
%
%If $\l\leq 1$, then there are not traveling waves solutions to \eqref{eq waves} connecting $0$ to any $U\in L^\infty$, with $U\not\equiv 0$
\end{theorem}
In the homogeneous case, Theorem \eqref{th waves} was proven by Diekmann and Thieme independently in \cite{Di1, T1}. In this specific case, $c^\star(e)$ does not depend on $e$, the equation is isotropic.\\

%When it comes to proving the non-existence of traveling waves, the situation is more involved. 
%For technical reasons, we focus on the specific case where the kernel $\Gamma$ is of the form
%\begin{equation}\label{sep}
%\Gamma(t,x,y) = K(x,y)e^{-\mu(y)t},
%\end{equation}
%with $\mu$ periodic and $K$ periodic and compactly supported in the sense that there is $A>0$ such that
%$K(x,y) =0$ for $\vert x -y \vert \geq A$. This includes the case where the %kernel is given by \eqref{gamma 2}. In this case, $\rho_0 = +\infty$.

We now turn to the non-existence of traveling waves with speed lesser than $c^\star(e)$. Such results are usually more technical to prove. A standard approach could be to prove that any solution of the the renewal equation \eqref{eq} asymptotically spreads with speed $c^\star(e)$ in the direction $e$, and to conclude by comparison. This strategy was used by Diekmann in \cite{Di2} in the homogeneous framework. However, here, we adopt a different approach, inspired by the study of heterogeneous reaction-diffusion equations, that consists in using some arguments from complex analysis to build adequate subsolutions.

To prove the next result, we assume, in addition to all the hypotheses above, that there are $T,C>0$ such that
\begin{equation}\label{non exp}
\Gamma(t_2,x,y) \leq C \Gamma(t_1,x,y), \quad \forall  t_2, t_1>0, \ t_2 \geq t_1 + T, \ x,y \in \R^N.
\end{equation}
\begin{theorem}\label{th non}
Assume that $\Gamma$ satisfies, in addition to the hypotheses above, \eqref{1 hyp waves per} and \eqref{non exp} and that  $\l >1$.

Then, for every $c\in [0,c^{\star}(e))$, for every $U>0$, $U\in C^0_{per}$, there are no generalized traveling waves in the direction $e$, with speed $c$, connecting $0$ to $U$.
\end{theorem}

\subsubsection{Application to a SIR model}\label{sec res 2}

Let us now present an application of our results to the SIR system \eqref{SIR 2}. We mention that other models from epidemiology or biology could be studied this way, see for instance \cite{TZ} for some (homogeneous) examples.\\

When considering the system \eqref{SIR 2}, we always assume that 
$\mu, \alpha \in C^0_{per}(\R^N)$, $K \in C^0(\R^N\times \R^N)$, with $\alpha,\mu,K>0$ and $K$ is periodic in the sense that $K(\cdot +k,\cdot+k) = K$ for all $k\in \Z^N$, symmetric in the sense that $K(x,y)=K(y,x)$ for all $x,y\in \R^N$, and decays faster than any exponential in the sense that, for all $\rho>0$, there is $C>0$ such that
$$
 K(x,y)\leq Ce^{-\rho\vert x - y\vert}\quad \text{ for every } \ x, y \in \R^N.
$$

Let us start with defining adequate notions of propagation and fading out for the SIR system \eqref{SIR 2}.

\begin{definition}[Propagation for SIR models]\label{def prop SIR}
We say that the epidemic propagates in \eqref{SIR 2} with initial datum $(S_0,I_0)$ if the solution $(S(t,x),I(t,x))$ is such that $S(t,x)$ converges to some $S_{\infty}(x) \in L^{\infty}(\R^N)$ locally uniformly as $t$ goes to $+\infty$ and if
$$
\sup_{x\in \R^N}\left( S_{\infty}(x) - S_0(x) \right)<0.
$$
We say that the epidemic fades out if $S$ converges similarly so some $S_{\infty}$ such that 
$$
 \vert S_{\infty}(x) - S_0(x)\vert  \underset{\vert x \vert \to +\infty}{\longrightarrow} 0.
$$

%Then, because $S(t,x)$ decays with respect to $t$, it converges pointwise in $x\in \R^N$ as $t$ goes to $+\infty$ to some function $S_{\infty}(x) \leq S_0(x)$. We say that the epidemic propagates if, and only if,
%$$
%\limsup_{\vert x \vert \to +\infty}\left(\frac{S_{\infty}(x)}{S_0(x)}\right) < 1.
%$$

\end{definition}

In other terms, the epidemic propagates if and only if a non-negligible number of infections occur everywhere, even far away from the initial focus of infection. Of course, there will always be infections on the support of $I_0$.

\begin{definition}[Traveling waves for SIR models]\label{def tw SIR}
Let $S_{-\infty}, S_{+\infty} \in L^{\infty}(\R^N)$ be such that $\inf_{x\in\R^N}S_{\pm \infty} >0$. We say that $(S(t,x),I(t,x))$ is a traveling wave in the direction $e\in\mathbb{S}^{N-1}$ with speed $c> 0$ for \eqref{SIR 2} connecting $(S_{-\infty}(x),0)$ to $(S_{+\infty}(x),0)$ if it solves \eqref{SIR 2} for every $t\in \R$ and if
$$
\sup_{x\cdot e - ct \leq \delta}\vert S(t,x) - S_{+\infty}(x) \vert \underset{\delta \to -\infty}{\longrightarrow} 0 \ \text{ and }\ \sup_{x\cdot e - ct \geq\delta }\vert S(t,x) - S_{-\infty}(x) \vert \underset{\delta \to +\infty}{\longrightarrow} 0,
$$
and
$$
\sup_{x\cdot e - ct \leq \delta} I(t,x)  \underset{\delta \to -\infty}{\longrightarrow} 0 \ \text{ and }\ \sup_{x\cdot e - ct \geq\delta } I(t,x)  \underset{\delta \to +\infty}{\longrightarrow} 0.
$$
\end{definition}

Then, we have the following result concerning the heterogeneous SIR model \eqref{SIR 2}.

\begin{theorem}\label{th sir 1}
Consider the SIR system \eqref{SIR 2} with $\mu, \alpha, K$ satisfying the above hypotheses. Take $S_0 \in C^0_{per}(\R^N)$ positive and let $\l$ be the principal periodic eigenvalue of the operator
 \begin{equation}\label{eq op SIR}
 \phi \mapsto \int_{\R^N}\frac{\alpha(x)S_0(y)}{\mu(y)}K(x,y) \phi(y)dy,
  \end{equation}
 acting on $C^0_{per}(\R^N)$.
 Then
 \begin{itemize}
 \item If $\l>1$, then, for every $I_0 \not\equiv 0$, $I_0\geq 0$ compactly supported, the epidemic propagates for the initial datum $(S_0,I_0)$. In addition, $S(t,x) \to S_{\infty}(x) \in C^0_{per}(\R^N)$, loc. unif. in $x\in \R^N$ as $t$ goes to $+\infty$, and
 $$
 \vert S_{\infty}(x) - \mc S(x) \vert \underset{\vert x \vert \to +\infty}{\longrightarrow} 0,
 $$
where $\mc S = S_0 e^{-U}$, with $U$ the unique positive bounded solution of
 \begin{equation}\label{f s sir}
 U(x) =   \int_{\R^N}\frac{\alpha(x)S_0(y)}{\mu(y)}K(x,y) (1- e^{-U(y)})dy.
 \end{equation}

 \item If $\l\leq 1$, the epidemic fades out for every $I_0$ compactly supported.

 \end{itemize}
 
 \end{theorem}
 
 Observe that we use the same notation for the principal periodic eigenvalue $\l$ of \eqref{eq op SIR} and \eqref{L}. This is because they coincide when $\Gamma$ is given by \eqref{gamma 2}.\\

 This theorem proves that the heterogeneous SIR system \eqref{SIR 2} satisfies a threshold phenomenon. It also indicates what the final population looks like far away from the initial focus of infection. Observe that the final repartition does not depend on $I_0$ far away from the initial focus of infection.\\

 Let us now consider the existence/non-existence of waves.
 \begin{theorem}\label{th sir 2}
 Consider the SIR system \eqref{SIR 2} with $\mu, \alpha, K$ satisfying the above hypotheses. For every $S_{-\infty} \in C_{per}^0(\R^N)$, $S_{-\infty} >0$, let $\l(\rho,c,e)$ be the principal periodic eigenvalue of the operator
  $$
 \phi \mapsto \int_{\R^N}\frac{\alpha(x)S_{-\infty}(y)}{\mu(y)+\rho c}K(x,y)e^{-\rho (y-x)\cdot e} \phi(y)dy,
 $$
 acting on $C^0_{per}(\R^N)$, and let $c^{\star}$ be given by \eqref{def c}.
 % \begin{itemize}
% \item If $\l(0,0,e)  < 1$, there are no traveling waves solutions to \eqref{SIR 2}.
 
% \item 

Assume that $\lambda_1> 1$, and let $U$ be the unique positive bounded solution of \eqref{f s sir} with $S_0$ replaced by $S_{-\infty}$.

Then, there are traveling waves solutions to \eqref{SIR 2} connecting $(S_{-\infty},0)$ to $(S_{-\infty}e^{-U}, 0)$ with speed $c$ in the direction $e\in \mathbb{S}^{N-1}$ for every $c>c^{\star}(e)$.

On the other hand, for any $e\in\mathbb{S}^{N-1}$, for any $S_{+\infty}>0$ periodic, for any $c \in [0,c^{\star}(e))$, there are no traveling waves connecting $(S_{-\infty},0)$ to $(S_{+\infty}, 0)$.
\end{theorem}
The paper is organized as follows: in Section \ref{sec th}, we study equation \eqref{eq}. We present some technical results in Section \ref{sec th spec}, where we study the operator $L$ defined by \eqref{L}. We prove the threshold phenomenon, Theorem \ref{th threshold}, in Section \ref{sec threshold}. Section~\ref{sec tw} is dedicated to the study of equation \eqref{eq waves}. In Section \ref{sec ex}, we prove the existence of traveling waves, Theorem \ref{th waves}, and we prove the non-existence result Theorem \ref{th non} in Section \ref{sec non}. Finally, we apply our results to the the case of the SIR system \eqref{SIR 2} in Section \ref{sec comp}.

%We explain how the SIR systems \eqref{SIR 1}, \eqref{SIR 2} can be written under the form of the renewal equations \eqref{eq}, \eqref{eq waves} in the Appendix \ref{app SIR}, and we explain how to apply our results to these SIR systems in order to get Corollaries~\ref{cor 1} and \ref{cor 2} in Appendix \ref{app cor}.\nota{du coup on vire les appendices, et on met une section}

\section{The threshold phenomenon}\label{sec th}

This section is dedicated to the proof of Theorem \ref{th threshold}, that states that $\l$, the principal periodic eigenvalue of the operator $L$, defined by \eqref{L}, characterizes the long-time behavior of \eqref{eq}. The existence of $\l$ is given by the Krein-Rutman theorem.

For notational simplicity, we assume in the sequel that $g^\prime(0)=1$ in the definition of \eqref{L}. This can be done without loss of generality by replacing $\Gamma$ and $V$ by $g^\prime(0)\Gamma$, $g^\prime(0)V$.

\begin{theorem}[Krein-Rutman theorem, \cite{KR}]\label{th KR}
Let $E$ be a real Banach space ordered by a salient cone $K$ (i.e., $K\cap (-K) = \{ 0\} $) with non-empty interior. Let $L$ be a linear compact operator. Assume that $L$ is strongly positive (i.e., $L(K\backslash \{ 0\}) \subset \inter K$). Then, there exists a unique eigenvalue $\lambda_1$ associated with some $u_1 \in K \backslash \{0\}$. Moreover, for any other eigenvalue $\lambda$, there holds
$$
\lambda_1 > \Re (\lambda).
$$
\end{theorem}
The Krein-Rutman theorem applies to the operator $L$ defined by \eqref{L}, on the Banach space $C_{per}^0(\R^N)$ (endowed with the $L^{\infty}$ norm) with $K$ being the cone of non-negative functions $K := \{\phi \in C^0_{per} \ : \ \phi \geq 0 \}$. The operator $L$ is linear and compact, owing to hypothesis \eqref{hyp gamma 1}. Indeed, it is readily seen that, for every $\e>0$, we can find $\delta>0$ such that, if $\vert x_1 - x_2 \vert \leq \delta$, we have, for every $\phi \in C_{per}^0(\R^N)$:
$$
\vert L\phi(x_1) - L\phi(x_2) \vert \leq \e \| \phi\|_{L^{\infty}}.
$$
This implies that the image of any bounded set of $C_{per}^{0}(\R^N)$ by $L$ is equicontinuous, and the Ascoli-Arzel\`a theorem, see \cite{B}, yields the compactness of $L$. The strong positivity of $L$ is readily seen: indeed, assume that there were $\phi \geq 0$, $\phi\not\equiv 0$ such that $L\phi(x_0)=0$ for some $x_0\in \R^N$. Then, because $V(x,y) >0$ for $x,y \in \R^N$ such that $\vert x -y\vert \leq r$, where $r$ is from the hypotheses in Section \ref{sec hyp}, we see that we should have $\phi(x) =0$ on $B_r(x_0)$. Iterating this argument, we would find that $\phi \equiv 0$; this proves the strong positivity of $L$.

\subsection{Approximation of the principal eigenvalue $\l$.}\label{sec th spec}

This section is dedicated to the proof of the following technical proposition:
\begin{prop}\label{approx}
Let $\l$ be the principal periodic eigenvalue of $L$. For every $\e>0$, there is $R_0 >0$ such that, for every $R>R_0$, there is $\phi_\e \in C^0(\R^N)$, strictly positive in $B_R$ and equal to zero elsewhere, such that
$$
\forall x\in \R^N, \quad L(\phi_\e)(x) \geq (\lambda_1 -\e)\phi_\e(x).
$$
\end{prop}
To prove this result, we introduce a family of operators $(L_R)_{R>0}$ whose principal eigenvalues will approximate $\l$:
\begin{equation}\label{LR}
L_R\phi(x) := \int_{B_R} V(x,y)\phi(y)dy.
\end{equation}
The operator $L_R$ acts on the Banach space $C^{0}(\ol B_R)$. Arguing as above, we can apply the Krein-Rutman theorem \ref{th KR} to $L_R$ on the Banach space $C^0(\ol B_R)$, 
%endowed with the $L^{\infty}(B_R)$ norm, and with positive cone the set of functions strictly positive on $\ol B_R$,
 to get the existence of its principal eigenvalue, that we call $\lambda_R$. We let $\phi_R\in C^0(\ol B_R)$ denote a principal eigenfunction, $\phi_R>0$ on $\ol{B_R}$. Let us observe that $\lambda_R$ is characterized by a Rayleigh-Ritz formula.
\begin{lemma}\label{lem vp}
Let $\gamma(x) := \frac{\gamma_2(x)}{\gamma_1(x)}$, where $\gamma_1, \gamma_2$ are from \eqref{sym V}. Then, the principal eigenvalue $\lambda_R$ of $L_R$ is given by
\begin{equation}\label{RR}
\lambda_R = \sup_{\phi \in L^2_{\gamma}(B_R)}\frac{\int_{B_R}L_R\phi(x)\phi(x)\gamma(x)dx}{\int_{B_R}\phi^2(x)\gamma(x)dx},
\end{equation}
where $L^2_{\gamma}(B_R)$ is the space $L^2(B_R)$ endowed with the scalar product
$$
\langle f, g\rangle_{L^2_{\gamma}} := \int_{B_R} f(x)g(x) \gamma(x)dx.
$$
\end{lemma} 
\begin{proof}
Owing to the hypothesis \eqref{sym V}, the operator $L_R$ is self-adjoint on the space $L^{2}_\gamma(B_R)$. Moreover, it is compact (it is a Hilbert-Schmidt operator, owing to hypothesis \eqref{hyp gamma vp}, see \cite{B}). Therefore, we can apply the spectral theorem, and the usual Rayleigh formula gives us that $\t \lambda$, the largest eigenvalue of $L_R$ (on $L^2_{\mu}(B_R)$), is given by
\begin{equation}\label{RR tilde}
\t \lambda = \sup_{\phi \in L^2_{\gamma}(B_R)}\frac{\int_{B_R}L_R\phi(x)\phi(x)\gamma(x)dx}{\int_{B_R}\phi^2(x)\gamma(x)dx}.
\end{equation}
It is readily seen that $\t \lambda \geq \lambda_R >0$ (the strict positivity comes from the fact that $V\geq0$, $V\not\equiv 0$). Let $\t\phi \in L^2_{\gamma}(B_R)$ be an eigenfunction associated with $\t \lambda$. Up to considering $\vert \t \phi\vert$, we assume that $\t \phi \geq 0$ almost everywhere. Hypothesis \eqref{hyp gamma vp} yields that $ \t\phi$ is  bounded, because $\t\lambda>0$ and
$$
\forall x \in \ol B_R, \quad \t \lambda^2 \t\phi^2(x) \leq \sup_{z\in B_R}\left(\int_{B_R}V^2(z,y)dy\right) \left(\int_{B_R}{\t\phi}^2(y)dy\right).
$$
Now, hypothesis \eqref{hyp gamma 1} yields that $\t \phi\in C^0(\ol B_R)$. The uniqueness (up to multiplication by a scalar) of the principal eigenvalue given by the Krein-Rutman theorem \ref{th KR} implies that $\t \lambda = \lambda_R$, hence the result.
\end{proof}

%%%%%%%%%%%%%%%%%%%%%%%%%%%%%%%%%%%%%%%%

We now prove that the sequence of principal eigenvalues $(\lambda_R)_{R>0}$ converges to the periodic principal eigenvalue $\l$.
\begin{prop}\label{cv vp}
The sequence $(\lambda_R)_{R>0}$ is increasing and it converges to $\l$ as $R$ goes to $+\infty$.
\end{prop}
\begin{proof}
\emph{Step 1. The sequence $(\lambda_R)_{R>0}$ is increasing.}\\
Let $0<R<R'$ be fixed, and let $\lambda_R, \lambda_{R'}$ be the principal eigenvalues of the operators $L_R, L_{R'}$ respectively. Let $\phi_{R}, \phi_{R'}$ denote some associated positive principal eigenfunctions. Define
$$
M^{\star} := \min\{ M>0 \ : \ M\phi_{R'} \geq \phi_R\ \text{ on }\ B_R\}.
$$
Then, by continuity, there is $x_0 \in \ol B_R$ such that $M^\star\phi_{R'}(x_0) = \phi_R(x_0)$. Hence,
$$
\lambda_R \phi_R(x_0) = L_{R}\phi_R(x_0) \leq L_{R}(M^{\star}\phi_{R'})(x_0) < M^\star L_{R'}\phi_{R'}(x_0) = \lambda_{R'}M^\star \phi_{R'}(x_0).
$$
The strict inequality comes from the fact that $V\geq 0$, $V\not\equiv 0$. This implies that
$$
\lambda_R <\lambda_{R'}. 
$$
We prove similarly that
$$
\lambda_R <\lambda_1,
$$
where $\l$ is the principal periodic eigenvalue of the operator $L$ defined by \eqref{L}.

\medskip
\emph{Step 2. Convergence to $\l$.}\\
We let $\phi_p>0$ be a periodic principal eigenfunction of $L$ associated with the eigenvalue $\l$. Owing to the Rayleigh-Ritz formula \eqref{RR} for $\lambda_R$, we have, for every $R>0$,
$$
\lambda_R \geq \frac{\int_{B_R}L_R\phi_p(x)\phi_p(x)\gamma(x)dx}{\int_{B_R}\phi_p^2(x)\gamma(x)dx}.
$$
Let us prove that
\begin{equation}\label{conv vp}
\frac{\int_{B_R}L_R\phi_p(x)\phi_p(x)\gamma(x)dx}{\int_{B_R}\phi_p^2(x)\gamma(x)dx} \underset{R \to +\infty}{\longrightarrow} \frac{\int_{[0,1]^N}L\phi_p(x)\phi_p(x)\gamma(x)dx}{\int_{[0,1]^N}\phi_p^2(x)\gamma(x)dx} = \l.
\end{equation}
Because the sequence $(\lambda_R)_{R>0}$ is bounded from above by $\l$, proving \eqref{conv vp} will yield the result. Observe that, because $\phi_p$, $L\phi_p$ and $\gamma$ are periodic, we have
$$
\frac{1}{\vert B_R \vert}\int_{B_R}\phi_p^2(x)\gamma(x)dx \underset{R\to+\infty}{\longrightarrow} \int_{{[0,1]^N}}\phi_p^2(x)\gamma(x)dx
$$
and
$$
\frac{1}{\vert B_R \vert}\int_{B_R}L\phi_p(x)\phi_p\gamma(x)dx \underset{R\to+\infty}{\longrightarrow} \int_{{[0,1]^N}}L\phi_p(x) \phi_p(x)\gamma(x)dx.
$$
Therefore, to prove \eqref{conv vp}, it is sufficient to show that
$$
\frac{1}{\vert B_R \vert}\left( \int_{B_R}(L_R\phi_p(x) - L\phi_p(x))\phi_p(x)\gamma(x)dx \right)\underset{R \to +\infty}{\longrightarrow} 0,
$$
or, equivalently, that
$$
\frac{1}{\vert B_R \vert} \int_{x\in B_R}\int_{y\in B_R^c}V(x,y)\phi_p(y)\phi_p(x)\gamma(x)dxdy\underset{R \to +\infty}{\longrightarrow} 0.
$$
We have (we let $C>0$ denote an arbitrary constant, independent of $R$)
\begin{equation*}
\begin{array}{llc}
\frac{1}{\vert B_R \vert} \int_{x\in B_R}\int_{y\in B_R^c}V(x,y)\phi_p(y)\phi_p(x)\gamma(x)dxdy\\
\leq \frac{C}{\vert B_R \vert} \int_{x\in B_{R-\sqrt{R}}}\int_{y\in B_R^c}V(x,y)dxdy + \frac{C}{\vert B_R \vert} \int_{x\in B_R \backslash B_{R-\sqrt{R}}}\int_{y\in B_R^c}V(x,y)dxdy \\
\leq C\sup_{\vert x \vert \leq R-\sqrt{R}}\left( \int_{y\in B_R^c}V(x,y)dy\right) + C\frac{\vert B_R \backslash B_{R-\sqrt{R}} \vert}{\vert B_R \vert} \sup_{x\in\R^N}\left(\int_{y\in\R^N}V(x,y)dy\right) \\
\leq C\sup_{\vert x \vert \leq R-\sqrt{R}}\left( \int_{y\in B_R^c}V(x,y)dy\right) + \frac{C}{\sqrt{R}} \sup_{x\in\R^N}\left(\int_{y\in \R^N}V(x,y)dy\right).
\end{array}
\end{equation*}
To conclude, let us show that $\sup_{\vert x \vert \leq R-\sqrt{R}}\left( \int_{y\in B_R^c}V(x,y)dy\right)$ goes to zero as $R$ goes to $+\infty$. If this were not the case, we could find $\e>0$ and a sequence $(x_n)_{n\in\N}$ such that $\vert x_n \vert \leq n-\sqrt{n}$ for every $n\in \N$ and
\begin{equation}\label{abs int}
 \liminf_{n\to+\infty}\int_{y\in B_n^c}V(x_n,y)dy >\e.
\end{equation}
We can define a sequence $(k_n)_{n\in\N} \in \Z^N$ such that $x_n - k_n \in [0,1)^{N}$ for every $n\in \N$. Hence, owing to the hypotheses \eqref{hyp gamma 1}, \eqref{hyp per} we get
$$
\int_{y\in B_n^c}V(x_n,y)dy = \int_{y\in B_n^c(-k_n)}V(x_n-k_n,y)dy \leq \int_{\vert y \vert \geq \frac{\sqrt{n}}{2}}V(x_n-k_n,y)dy\underset{n\to+\infty}{\longrightarrow}0,
$$
which contradicts \eqref{abs int}. This proves the convergence and concludes the proof.
\end{proof}

We can now turn to the proof of Proposition \ref{approx}. We mention that a similar result was obtained by H. Berestycki, J. Coville and H.-H. Vo in \cite{BCV} in the context of non-local reaction-diffusion equations, however, the situation considered here allows for a simpler proof.
\begin{proof}[Proof of Proposition \ref{approx}]
Let $\e>0$ be fixed. Let $\phi_p>0$ be a positive principal periodic eigenfunction of $L$. Owing to Proposition \ref{cv vp}, we can find $R>0$ large enough so that $\lambda_R > \l -\frac{\e}{2}$, where $\lambda_R$ is the principal eigenvalue of the operator $L_R$. Let $\phi_R$ be a positive principal eigenfunction of $L_R$ associated with $\lambda_R$. For $\eta>0$, to be determined after, let $\chi_R\leq 1$ be a continuous function such that $\chi_R>0$ on $B_R$, $\chi_R = 1$ on $B_{R-\eta}$, and $\chi_R =0$ on $B_{R}^c$. We define
$$
\phi_\e(x) := \begin{cases} \phi_R(x)\chi_{R}(x) \quad &\text{ for }\ x\in B_R, \\ 0 \quad &\text{ for }\ x\in B^c_R. \end{cases}
$$
The function $\phi_\e$ is continuous on $\R^N$, strictly positive in $B_R$, zero elsewhere and compactly supported.
For $x\in B_R$, we have
$$
\begin{array}{rl}
L(\phi_\e)(x) &= \int_{B_R}V(x,y)\phi_R(y)\chi_R(y)dy \\
&\geq \lambda_R \phi_R(x) -  \| \phi_R\|_{L^\infty}\left(\int_{B_R \backslash B_{R-\eta}}V(x,y)dy\right) \\
%&\geq \int_{B_R}K(x,y)\phi_R(y)dy - \| K \|_{L^\infty} \| \phi_R\| \vert B_R \backslash B_{R-\eta}\vert
&\geq (\lambda_R-\frac{\e}{2}) \phi_R(x)\chi_R(x) + \frac{\e}{2}\left(\min_{x\in \ol B_R}\phi_R(x)\right) \\ &\quad \quad \quad\quad\quad \quad \quad\quad \quad \quad-  \| \phi_R\|_{L^\infty}\left(\sup_{x\in B_R}\int_{B_R \backslash B_{R-\eta}}V(x,y)dy\right).
\end{array}
$$
Therefore,for $\eta$ small enough, independent of $x$, we have
$$
\forall x \in B_R, \quad L(\phi_\e)(x) \geq (\l -\e) \phi_\e(x).
$$
For $x\in B^c_R$, this inequality is readily verified, hence the result.
\end{proof}

\subsection{Long-time behavior of solutions of \eqref{eq}}\label{sec threshold}

This section is dedicated to the proof of Theorem \ref{th threshold}. For convenience, we let $T$ denote the nonlinear operator
\begin{equation}\label{def T}
(Tu)(x) := \int_{\R^N}V(x,y)g(u(y))dy.
\end{equation}
The operator $L$, defined by \eqref{L}, is the linearization of $T$. We start with a technical lemma.

\begin{lemma}\label{inf}
Assume that $\l >1$, where $\l$ is the principal periodic eigenvalue of~$L$. Let $u\in C^{0}(\R^N)$, $u>0$, be such that
$$
T(u)\leq u.
$$
Then
$$
\inf_{\R^N} u >0.
$$
\end{lemma}
\begin{proof}
Assume that $\l>1$ and that $u\in C^0(\R^N)$, $u>0$, is such that $Tu \leq u$. Let $\e>0$ be such that $\e<\l-1$. Owing to Proposition \ref{approx}, we can find $R>2\sqrt{N}$ ($N$ is the space dimension) and $\phi_\e \in C^0(\R^N)$, $\phi_\e > 0$ on $B_R$, $\phi_\e = 0$ elsewhere, such that
$$
(\l -\e)\phi_\e \leq L(\phi_\e).
$$
We define
$$
\eta^\star := \max\{ \eta >0 \ : \ u \geq \eta \phi_\e\}.
$$
We start to assume that $\eta^\star$ is such that
\begin{equation}\label{eta star}
\eta^\star \leq \frac{1}{C\|\phi_\e \|_{L^\infty}},
\end{equation}
where $C>0$ stands in the whole proof for the constant from \eqref{hyp g}.

By definition of $\eta^{\star}$, and by continuity, $u \geq \eta^\star \phi_\e$ and there is a contact point $x_0 \in B_R$ such that $u(x_0) = \eta^\star \phi_\e(x_0)$. Owing to the hypothesis \eqref{hyp g}, we have
$$
\begin{array}{llc}
T(\eta^\star \phi_\e) &\geq L(\eta^\star \phi_\e) - CL({\eta^\star}^2\phi_\e^2) \\
&\geq \eta^\star(1-C\eta^\star \|\phi_\e \|_{L^\infty})L(\phi_\e)\\
&\geq \eta^\star(1-C\eta^\star \|\phi_\e \|_{L^\infty})(\l-\e)\phi_\e.
\end{array}
$$
Therefore, because $T$ is order-preserving,
$$
\eta^\star(1-C\eta^\star \|\phi_\e \|_{L^\infty})(\l -\e)\phi_\e \leq T(\eta^\star \phi_\e)\leq T(u) \leq u.
$$
Observe that it is not possible to have
$$
1 \leq (1-C\eta^\star \|\phi_\e \|_{L^\infty})(\l -\e).
$$
Indeed, this would imply that
$$
\eta^\star \phi_\e \leq T(\eta^\star \phi_\e)\leq T(u) \leq u,
$$
and then, evaluating at $x_0$, we would get $T(\eta^\star \phi_\e)(x_0) = T(u)(x_0)$, which would yield, owing to the non-negativity of $V$, that $g(\eta^\star \phi_\e) \equiv g(u)$. Because of the hypotheses on $g$, this would imply that $\eta^\star \phi_\e \equiv u$, which is not possible because $u$ is positive on $\R^N$, while $\phi_\e$ is compactly supported. Therefore
$$
(1-C\eta^\star \|\phi_\e \|_{L^\infty})(\l -\e) < 1,
$$
i.e.,
$$
\frac{\l - 1 -\e}{C\|\phi_\e \|_{L^\infty}(\l -1)}<\eta ^\star.
$$
In other terms, we have proven that, if $\eta^{\star}$ satisfies \eqref{eta star}, then $\eta^{\star}$ is greater than a positive constant independent of $u$. Clearly, this is also the case if \eqref{eta star} is not verified: in this case, we have directly $\eta^{\star} \geq \frac{1}{C\|\phi_\e\|_{L^{\infty}}}$. To sum up, in both cases, we have proven that there is $\kappa >0$, independent of $u$, such that
$$
\kappa\phi_\e(x) \leq u(x), \quad \text{ for } x\in \R^N.
$$
%$$
Because we took $R>2\sqrt{N}$, we have $ [0,1]^N \subset B_{\frac{R}{2}}$, hence
\begin{equation}\label{est cell}
\forall x\in [0,1]^N, \quad 0<\kappa\left(\min_{B_{\frac{R}{2}}}\phi_\e\right) \leq u(x).
\end{equation}
Because $\kappa$ is independent of $u$, we can apply \eqref{est cell} to $u(\cdot + k)$, for any $k\in \Z^N$, to find that
$$
\forall x\in \R^N, \quad 0<\kappa\left(\min_{B_{\frac{R}{2}}}\phi_\e\right) \leq u(x),
$$
hence the result.
\end{proof}

We now prove that $\l$ characterizes the existence of non-null solutions to \eqref{eq stat}.

\begin{prop}\label{sol stat}
Let $\l$ be the principal periodic eigenvalue of the operator $L$.
\begin{itemize}
\item If $\l > 1$, the equation \eqref{eq stat} has a unique non-negative, non-zero bounded solution. Moreover, this solution is periodic.

\item If $\l \leq 1$, the equation \eqref{eq stat} has no non-negative non-zero bounded solutions.
\end{itemize}
\end{prop}
\begin{proof}
\emph{Case $\l >1$. Existence of a non-zero periodic solution.} \\
Let $\phi_p>0$ be a positive principal periodic eigenfunction associated to $\l$. For $\e >0$, we have, owing to the hypothesis \eqref{hyp g},
$$
T(\e \phi_p) \geq L(\e\phi_p) - CL(\e^2 \phi_p^2).
$$
Because
$$
L(\phi_p^2) \leq \| \phi_p\|_{L^\infty} L(\phi_p),
$$
we find that, up to taking $\e$ small enough, we have
$$
T(\e \phi_p) \geq \e(1 - C\e \| \phi_p\|_{L^{\infty}}) \l \phi_p \geq \e\phi_p.
$$
We now define a sequence of positive, continuous periodic functions $(U_n)_{n\in\N}$ by 
\begin{equation}\label{Un}
U_{n+1} = T(U_n), \quad U_0 = \e \phi_p.
\end{equation}
Because $U_1 = T(U_0) \geq U_0$ and because $T$ is order-preserving, it is readily seen that the sequence $(U_n)_{n\in \N}$ is non-decreasing. Moreover, it is bounded independently of $n\in \N$ by $\|g\|_{L^{\infty}}\sup_{x\in \R^N}\left(\int_{\R^N}V(x,y)dy\right)$, therefore it converges pointwise as $n$ goes to $+\infty$ to some periodic function $U$. In addition, because $ U \geq U_{0}$, the function $U$ is not everywhere equal to zero. The uniform boundedness of the sequence together with hypothesis \eqref{hyp gamma 1} yield that $(U_n)_{n\in\N}$ is locally equicontinuous. The Ascoli-Arzel\`a theorem then implies that the convergence of $U_n$ to $U$ is locally uniform. An easy computation yields that, for every $x\in \R^N$, $T(U_n)(x)$ converges to $T(U)(x)$. Taking the limit $n\to+\infty$ in \eqref{Un}, we find that $U$ is a periodic, positive, continuous solution of \eqref{eq stat}.

\medskip
\emph{Case $\l >1$. Uniqueness of the positive solution.}\\
Let $U$ be the positive continuous periodic solution of \eqref{eq stat} given by the first step. Let~$\t U$ be a bounded non-negative, non-zero solution, not necessarily periodic. Let us prove that $U \equiv \t U$. First, observe that $\t U$ is continuous owing to \eqref{hyp gamma 1}. Moreover, Lemma~\ref{inf} yields that the infimum of $\t U$ is positive. Therefore, we can define
$$
\eta^\star := \max\{ \eta >0 \ : \ \t U \geq \eta U \} >0.
$$
It is sufficient to prove that $\eta^\star \geq 1$. Indeed, this will imply that $\t U\geq U$, and inverting the roles of $\t U$ and $U$ will yield the equality between the two solutions. We argue by contradiction, we assume that $\eta^\star<1$.

By continuity, we have that $\eta^\star U \leq \t U$, and there is a sequence $(x_n)_{n\in\N}$ such that $\eta^\star U(x_n) - \t U(x_n) \to 0$ as $n\to +\infty$. Because the operator $T$ is order-preserving, we have
$$
T(\eta^\star U) \leq T(\t U) = \t U.
$$
Because we assume that $\eta^\star <1$, the hypothesis \eqref{hyp g concave} implies that $\eta^\star g(z) \leq g(\eta^\star z)$, for every $z >0$. Hence, $\eta^\star T(U) \leq T(\eta^\star U)$, and then
\begin{equation}\label{eg}
\forall x\in \R^N, \quad \eta^\star U(x) = \eta^\star T(U)(x) \leq T(\eta^\star U)(x) \leq T(\t U)(x) =\t U(x).
\end{equation}
We let $(k_n)_{n\in\N}$, $(z_n)_{n\in\N}$ be such that $x_n = k_n + z_n$, with $k_n \in \Z^N$ and $z_n \in [0,1)^N$. Up to extraction, we find $z\in [0,1]^N$ such that $z_n \to z$ as $n$ goes to $+\infty$. We define the sequence of translated functions
$$
\t U_n := \t U(\cdot+k_n).
$$
The periodicity hypothesis \eqref{hyp per} yields that $T(\t U_n) = \t U_n$. The sequence $(\t U_n)_{n\in \N}$ is bounded independently of $n$ (because $\t U$ is bounded). Therefore, owing to hypothesis \eqref{hyp gamma 1}, the sequence $(\t U_n)_{n\in \N}$ is equicontinuous, hence we can apply the Ascoli-Arzel\`a theorem to find that, up to extraction, $\t U_n$ converges locally uniformly as $n$ goes to~$+\infty$ to some $\t U_{\infty}$. Owing to Lemma \ref{inf}, $\t U_{\infty} \not\equiv 0$. Evaluating \eqref{eg} at $x + k_n$ and taking the limit $n\to +\infty$, we find that
$$
\forall x\in \R^N, \quad \eta^\star U(x) = \eta^\star T(U)(x) \leq T(\eta^\star U)(x) \leq T(\t U_\infty)(x) =\t U_\infty(x).
$$
Moreover, we have
$$
\eta^\star U(z) =\t U_\infty(z).
$$
Arguing as above, we find that this yields
$$
\forall x\in \R^N, \quad g(\eta^\star U(x))=\eta^\star g(U(x)).
$$
Owing to the hypothesis \eqref{hyp g concave}, this is impossible because $U\not\equiv0$ and $\eta^\star<1$. We have reached a contradiction, hence $\t U \equiv U$.

\medskip
\emph{Case $\l \leq 1$. Non-existence of positive solutions.}\\
Assume that $\l\leq1$ and that there is a bounded solution $U \geq 0$ of \eqref{eq stat}. Let $\phi_{p}$ be a principal periodic eigenfunction associated with $\l$. Because $U$ is bounded and because $\inf_{\R^N}\phi_p >0$, we can define
$$
M^{\star} := \inf\{ M \geq 0 \ : \ M \phi_p \geq U \}.
$$
Then, by continuity, $U \leq M^{\star}\phi_p$ and there is a sequence $(x_n)_{n\in\N}\in \R^N$ such that $M^\star \phi_p(x_n) - U(x_n)\to 0$ as $n$ goes to $+\infty$. Owing to \eqref{hyp g}, we have
\begin{equation}\label{eg no}
U = T(U) \leq T(M^\star \phi_p) \leq L(M^\star \phi_p) = \lambda_1 M^\star \phi_p \leq M^\star \phi_p.
\end{equation}
We define two sequences $(k_n)_{n\in\N}$ and $(z_n)_{n\in\N}$ as in the previous step, that is, $x_n = k_n + z_n$, with $k_n \in \Z^N$ and $z_n \in [0,1)^N$. We let $z\in [0,1]^N$ be a limit, up to extraction, of $z_n$ as $n$ goes to $+\infty$. Evaluating \eqref{eg no} at $x=x_n$ and taking the limit $n \to +\infty$, we find that, up to extraction,
$$
T(M^\star \phi_p)(z) = L(M^\star \phi)(z),
$$
which implies that (remember that we assume $g^\prime(0)=1)$
$$
M^\star\phi_p \equiv g(M^\star \phi_p).
$$
Owing to hypothesis \eqref{hyp g concave}, this is possible only if $M^\star =0$, that is, if $U\equiv 0$.
\end{proof}
We are now in position to prove Theorem \ref{th threshold}.
\begin{proof}[Proof of Theorem \ref{th threshold}]
Let $u$ be the solution of \eqref{eq}. Owing to Proposition \ref{prop cv}, it converges to $u_{\infty}$, solution of \eqref{eq limit}. Assume that $\l>1$, and let $U$ be the unique positive periodic solution of \eqref{eq stat} given by Proposition \ref{sol stat}. We take a diverging sequence $(x_n)_{n\in\N}\in\R^N$  such that
\begin{equation}\label{limit inf translations}
\vert u_{\infty}(x_n) - U(x_n)\vert \underset{n \to +\infty}{\longrightarrow} \lim_{R\to+\infty} \left(\sup_{\vert x \vert \geq R}\vert u_{\infty}(x) - U(x)\vert\right).
\end{equation}
We choose $(k_n)_{n\in\N}\in \Z^N$ and $(z_n)_{n\in\N} \in [0,1)^N$ such that $x_n = k_n + z_n$. Because $x_n$ diverges, so does $k_n$. Up to extraction, we assume that $z_n$ converges to some $z\in [0,1]^N$ as $n$ goes to $+\infty$.
We introduce the translated functions
$$
u_n := u_{\infty}(\cdot+k_n).
$$
Because $u_{\infty}$ is solution of \eqref{eq limit}, $u_n$ solves
\begin{equation}\label{eq limit n}
u_{n}(x) = \int_{\R^N}V(x,y)g(u_n(y))dy + f_{\infty}(x+k_n), \quad x\in \R^N.
\end{equation}
Observe that, because $f_\infty \geq 0$, we can apply Lemma \ref{inf} to get that there is $\kappa>0$ such that $u_n \geq \kappa$, for every $n\in \N$.

Because $f_\infty$ is bounded and uniformly continuous and owing to hypothesis \eqref{hyp gamma 1}, we find that the sequence $(u_n)_{n\in\N}$ is bounded and equicontinuous. Owing to the Ascoli-Arzel\`a theorem, we can extract a sequence that converges locally uniformly to some function $\t U$. We have $\t U \geq \kappa >0$, hence $\t U$ is not everywhere equal to zero. Moreover, because $\vert k_n\vert$ goes to $+\infty$ as $n$ goes to $+\infty$, $f_{\infty}(x+k_n)$ converges to $0$ locally uniformly as $n$ goes to $+\infty$, owing to hypothesis \eqref{hyp f 2}. Taking the limit $n\to+\infty$ in \eqref{eq limit n}, we find that $\t U$ is a bounded non-negative, non-zero solution of \eqref{eq stat}. Proposition~\ref{sol stat} then yields that $\t U \equiv U$, where $U$ is the unique positive periodic solution of \eqref{eq stat}.

Owing to \eqref{limit inf translations}, and using the fact that $U$ is periodic and that $u_n$ converges locally uniformly to $\t U$, we have
$$
\begin{array}{llc}
\lim_{R\to+\infty} \left( \sup_{\vert x \vert \geq R}\vert u_{\infty}(x) - U(x)\vert \right) &= \lim_{n\to+\infty}\vert u_{n}(z_n) - U(z_n)\vert \\
&= \vert \t U(z) - U(z)\vert\\
&=0.
\end{array}
$$
This proves the result when $\l>1$. When $\l \leq 1$, the proof is similar.
\end{proof}

We have now proven Theorem \ref{th threshold}: the principal periodic eigenvalue of $L$ characterizes the long-time behavior of solutions of \eqref{eq}. Our proof used the symmetry hypothesis \eqref{sym V}. As mentioned in Section \ref{sec res 1}, this hypothesis is somewhat necessary. Indeed, we have the following:
\begin{prop}\label{prop ce}
We consider the $1$-dimensional case. Define
\begin{equation}\label{ker}
\Gamma(t,x,y)= \frac{1}{\sqrt{4\pi t}}e^{-\frac{(x-y)^2}{4t}}e^{-(x-y)}e^{-\frac{3}{2}t} ,\quad t>0, \ x, y \in \R.
\end{equation}
We also fix $g(z) = 1 - e^{-z}$ and $f(t,x) =f_0(x) \in C^2(\R)$, where $f_0 \geq 0$, $f_0 \not\equiv 0$ is compactly supported.

Then, the solution $u$ of \eqref{eq} with $\Gamma,f,g$ as above does not propagate. However, the principal periodic eigenvalue $\lambda_1$ of $L$ defined by \eqref{L} is strictly greater than $1$ (it is equal to $2$).

In other terms, Theorem \ref{th threshold} does not hold for such kernels.
\end{prop}
Observe that the kernel $\Gamma$ in the theorem satisfies all the hypotheses needed for Theorem \eqref{th threshold}, except the symmetry one, \eqref{sym V}.
\begin{proof}
Let us start with proving that $u$, the solution of \eqref{eq} with such $\Gamma,f,g$ does not propagate. Observe that the kernel $\Gamma(t,x,y)$ is the fundamental solution of the operator
$$
\partial_t  - \partial_{xx}  - 2\partial_x  +\frac{1}{2}.
$$
Therefore, we have
$$
\partial_t(u-f_0)  - \partial_{xx}(u-f_0)  - 2\partial_x(u-f_0)  +\frac{1}{2}(u-f_0) = g(u),
$$
hence
\begin{equation}\label{eq rd u}
\partial_t u  = \partial_{xx} u + 2\partial_x u + (1 - e^{-u}) - \frac{1}{2}u + h,
\end{equation}
where $h = -\partial_{xx} f_0 - c\partial_x f_0  +\frac{1}{2}f_0$ is a continuous, compactly supported function.

Now, observe that, for $A>0$ large enough, the function
$$
s(x) = A e^{-x}
$$
is supersolution of \eqref{eq rd u}. Up to increasing $A$, we ensure also that $s\geq u(0,\cdot) = f_0$. The parabolic comparison principle implies that
$$
u(t,x) \leq s(x), \quad \forall t>0, \ x\in \R.
$$
Hence, we do not have propagation of the epidemic: $\lim_{x\to +\infty }\limsup_{t\to +\infty}u(t,x) =0 $.

On the other hand, the principal periodic eigenvalue of the operator $L$ is strictly greater than $1$. Indeed, in the case considered here,
$$
V(x,y) = \frac{1}{\sqrt{6}} e^{- \sqrt{\frac{3}{2}} \vert x - y\vert}e^{-(x-y)}.
$$
Observe that this is the Green function of the elliptic operator $- \partial_{xx}  - 2\partial_x  +\frac{1}{2}$. The function everywhere constant equal to $1$ is a principal periodic eigenfunction, and then we compute that $\lambda_1 =2$.This concludes the proof.
\end{proof}
Let us say a word about this result. It proves that hypothesis \eqref{sym V} can not be totally removed. Two questions then arise: what is the optimal condition on $\Gamma$ that could ensure that the principal periodic eigenvalue characterizes the propagation? For general $\Gamma$, can we find another criterion that ensures that we have propagation or fading out?

For the first question, observe that we used hypothesis \eqref{sym V} only one time, it was in the proof of Lemma \ref{lem vp}, and we used it only to say that the operator $L$ is conjugated to a symmetric operator acting on $C^0_{per}$. We leave it as an open question to find more general conditions.

Concerning the second question, finding a more general criterion, it is enlightning to observe that this phenomenon (the fact that the principal periodic eigenvalue does not characterizes the long-time behavior of the system when the problem is not symmetric) was already observed in the setting of reaction-diffusion equations, see \cite{BReigen} for instance. Our proof is an adaptation of this fact. However, when studying reaction-diffusion equations, a notion called \emph{generalized principal eigenvalue} has been introduced, by Berestycki, Nirenberg and Varadhan \cite{BNV} and later extended in \cite{BDR, BReigen}, and was used successfully to study non-symmetric reaction-diffusion equations. We leave it for later works to extend such a notion for integral operators.

\section{Traveling waves}\label{sec tw}

This section is dedicated the proofs of Theorems \ref{th waves} and \ref{th non}. We define the two following operators:
\begin{equation*}
\begin{cases}
\mathcal{T} u := \int_{\R^N}\int_{0}^{+\infty}  \Gamma(\tau,x,y)g(u(t-\tau,y)) d\tau dy, \\
\mc L u :=  \int_{\R^N}\int_{0}^{+\infty}  \Gamma(\tau,x,y)u(t-\tau,y) d\tau dy.
\end{cases}
\end{equation*}
We still assume, without loss of generality, that $g^\prime(0)=0$.
Owing to the hypothesis \eqref{hyp g}, the operator $\mc T$ is ``controlled'' by its linearization $\mc L$ in the sense that there is $C>0$ such that:
\begin{equation}\label{L KPP}
\mc Lu - C\mc L u^2 \leq \mc Tu \leq \mc Lu, \quad \text{ for all } \ u\geq0.
\end{equation}
With these notations, the equation \eqref{eq waves} for traveling waves rewrites $u = \mc T u$. We say that the function $u$ is a subsolution (resp. supersolution) of \eqref{eq waves} if is satisfies $u \leq \mc T u$ (resp. $u\geq \mc T u$).\\

We recall that, in this whole section, we assume that $\Gamma$ satisfies \eqref{1 hyp waves per}. This hypothesis may be actually relaxed, it is sufficient to assume that,
%there is $\rho_{0}>0$ such that, for every $\rho\in [0,\rho_0)$, 
for every $\rho\geq 0$, for every $e\in \mathbb{S}^{N-1}$ and for every $c\geq 0$, the kernel
\begin{equation}\label{hyp waves per}
\Gamma_{\rho,c,e}(\tau,x,y) := \Gamma(\tau,x,y)e^{-\rho(c\tau + (y-x)\cdot e)}
\end{equation}
satisfies hypothesis \eqref{hyp gamma 1}, locally uniformly in $(\rho,c)$. It is easy to check that \eqref{hyp waves per} is a consequence of \eqref{1 hyp waves per}, together with the periodicity hypothesis \eqref{hyp per} and with \eqref{hyp gamma 1}. In the homogeneous framework, a similar hypothesis was assumed by Diekmann \cite{Di1}.\\
%\footnote{The next results actually hold true by assuming that \eqref{hyp waves per} is verified only for $\rho \in[0,\rho_0)$, for some $\rho_0>0$. This would be important if we wanted to consider kernels $\Gamma$ that decay at most exponentially, which would be necessary if we wanted to apply our results to SIR models with diffusion in the $I$ individuals - however, we will not treat this case here.}\\ 

%\subsection{Linearized problem and the critical speed}\nota{enlever ce titre?}
The strategy of proof we employ here is inspired by some techniques developped for the study of KPP reaction-diffusion equations, see \cite{NR} for instance.

To build traveling waves solutions to \eqref{eq waves}, we will use a supersolution-subsolution algorithm. A key-point is the following computation: let $\rho,c >0$ and $e\in \mathbb{S}^{N-1}$ be chosen. Then, for $\phi \in C^{0}_{per}(\R^N)$, we have
\begin{equation}\label{calc waves}
\begin{array}{rll}
\mc L(\phi(x) e^{-\rho( x\cdot e -c t)}) &=& \int_{0}^{+\infty} \int_{\R^N}\Gamma(\tau,x,y) e^{-\rho( y \cdot e - c(t-\tau))} \phi(y)  dy  d\tau \\
&=&\left(   \int_{\R^N}  \left(   \int_{0}^{+\infty}  \Gamma(\tau,x,y)e^{-\rho c \tau} d\tau \right) e^{-\rho(y- x) \cdot e}  \phi(y) dy \right) e^{-\rho( x\cdot e -c t)} \\
&=& \mc S_{\rho,c,e}(\phi)e^{-\rho( x\cdot e -c t)},
%&=& \left(   \int_{\Omega}\frac{1}{\mu(y)+\lambda c}e^{-\lambda(x- y) \cdot e}V(x,y)   \phi(y)  dy \right) e^{-\lambda( x\cdot e -c t)}.
\end{array}
\end{equation}
where $\mc S_{\rho,c,e}$ is defined by \eqref{Slc}. For notational simplicity, we assume that the direction $e\in \mathbb{S}^{N-1}$ is fixed, and we omit it in the indices from now on. We recall that $\lambda_1(\rho,c)$ denotes the principal periodic eigenvalue of $\mc S_{\rho,c}$. We let $\phi_{\rho,c}$ be an associated positive principal periodic eigenfunction. It follows from the computation \eqref{calc waves} that
$$
\mc L(\phi_{\rho,c}(x) e^{-\rho( x\cdot e -c t)}) = \lambda_1(\rho,c)\phi_{\rho,c} e^{-\rho( x\cdot e -c t)}.
$$
Clearly, if $\l(\rho,c) \leq 1$, \eqref{L KPP} implies that $\phi_{\rho,c}(x)e^{-\rho(x\cdot e -ct)}$ is a supersolution of \eqref{eq waves}. We conclude these remarks with a technical result:
\begin{prop}\label{dec}
The function
$$
(\rho,c)\in  [0,+\infty)\times[0,+\infty) \mapsto \lambda_{1}(\rho,c) \in \R
$$
is continuous. Moreover, for $\rho >0$, the function
$$
c \in [0,+\infty) \mapsto \lambda_1(\rho,c) \in \R
$$
is strictly decreasing.
\end{prop}
\begin{proof}
The strict monotonicity of $c\mapsto \l(\rho,c)$ for $\rho>0$ can be proven exactly as in the proof of Proposition~\ref{cv vp}, Step~1, therefore we do not repeat it. The continuity follows from hypothesis \eqref{hyp waves per}. Indeed, take $(\rho,c)\in [0,+\infty)^2$ and two sequences $(\rho_n)_{n\in \N}, (c_n)_{n\in\N}$, where $\rho_n, c_n \geq 0$, such that $\rho_n \to \rho$, $c_n \to c$. Let $\phi_n$ denote the principal eigenfunction of $S_{\rho_n,c_n}$ normalized so that $\sup\phi_n = 1$. Letting $x_n, \t x _n$ be some points of $[0,1]^N$ where $\phi_n$ is respectively minimal and maximal, we see that
$$
\int_{\R^N}\int_{0}^{+\infty}\Gamma_{\rho_n,c_n}(\tau,x_n,y)d\tau dy \leq \l(\rho_n,c_n) \leq \int_{\R^N}\int_{0}^{+\infty}\Gamma_{\rho_n,c_n}(\tau,\t x_n,y)d\tau dy.
$$
We can assume that, up to extraction, the sequences $(x_n)_{n\in\N}, (\t x_n)_{n\in\N}$ converge to some $x_\infty, \t x_\infty \in [0,1]^N$. Owing to hypothesis \eqref{hyp waves per}, we have that $\int_{\R^N}\int_{0}^{+\infty}\Gamma_{\rho_n,c_n}(\tau,x_n,y)d\tau dy$ converges to $\int_{\R^N}\int_{0}^{+\infty}\Gamma_{\rho,c}(\tau,x_\infty,y)d\tau dy >0$ as $n$ goes to $+\infty$, and then we find that $(\lambda_1(\rho_n, c_n))_{n\in\N}$ is bounded from below by a positive constant. It is also bounded from above. Up to performing an extraction, we assume that it converges to some $\ol\lambda >0$. For every $n\in \N$, we have
$$
S_{\rho_n,c_n}(\phi_n) = \l(\rho_n,c_n)\phi_n.
$$
Owing to the hypothesis \eqref{hyp waves per} and to the normalization, we find that the sequence $(\phi_n)_{n\in \N}$ is equicontinuous. The Ascoli-Arzel\`a theorem gives us that  $\phi_n$ converges up to extraction  to some positive $\phi_{\infty} \in C^0_{per}(\R^N)$ that satisfies
$$
S_{\rho,c}(\phi_\infty) = \ol \lambda \phi_{\infty}.
$$
The uniqueness of the principal periodic eigenvalue implies that $\ol \lambda = \l(\rho,c)$, hence the result.
\end{proof}
%%
%
%\begin{proof}
%We argue by contradiction. Assume that there are $\lambda, c, \e >0$ such that
%\begin{equation*}
%k(\lambda, c+ \e) \geq \lambda_1(\rho,c).
%\end{equation*}
%We define
%$$
%\eta^\star := \sup \{ \eta >0 \ : \ \phi_{\rho,c} \geq \eta \phi_{\lambda,c+\e} \}.
%$$
%and
%$$
%w := \phi_{\rho,c} - \eta^{\star} \phi_{\lambda,c+\e}.
%$$
%By continuity and periodicity, $w$ is non-negative, and there is a contact point $x_0$ such that $w(x_0)=0$. Then, the same computation as in the proof of Proposition \ref{cv vp} yields that
%$$
%S_{\lambda,c}w < k(\lambda,c+\e)w,
%$$
%%
%%\begin{equation*}
%%\begin{array}{llc}
%%S_{\lambda,c}w &=  S_{\lambda,c} \phi_c - \eta S_{\lambda,c}\phi_{c+\e} \\
%%&= \lambda_1(\rho,c) \phi_c - \eta S_{\lambda,c+\e}\phi_{c+\e} -  \eta (S_{\lambda,c}- S_{\lambda,c+\e})\phi_{c+\e} \\
%%&= \lambda_1(\rho,c) \phi_c - \eta k(\lambda,c+\e)\phi_{c+\e} -  \eta (S_{\lambda,c}- S_{\lambda,c+\e})\phi_{c+\e} \\
%%&< \lambda_1(\rho,c) \phi_c - \eta k(\lambda,c+\e)\phi_{c+\e}\\
%%&< k(\lambda,c+\e)w.
%%\end{array}
%%\end{equation*}
%and valuating at the contact point $x_0$ yields that $S_{\lambda,c}w(x_0)<0$, which is impossible, hence the result.
%%Evaluating at the contact point $x_0$ yields that $S_{\lambda,c}w(x_0)=0$, and then $w\equiv 0$. Therefore, $\phi_{\rho,c} = \eta^\star \phi_{\lambda,c+\e}$. But this implies that
%%$$
%%0 < (S_{\lambda,c} - S_{\lambda,c+\e})\phi_{\rho,c} = (\lambda_1(\rho,c) - k(\lambda,c+\e))\phi_{\rho,c} \leq 0,
%%$$
%%which is the contradiction we sought, hence the result.
%\end{proof}

\subsection{Existence of traveling waves.}\label{sec ex}

The next result gives the existence of supersolutions and subsolutions to equation~\eqref{eq waves}.

\begin{prop}\label{prop sol}
Assume that $\lambda_1 >1$. Let $U$ denote the positive periodic solution of \eqref{eq stat} provided by Proposition \ref{sol stat}. For $c>c^{\star}(e)$, there are $0 \leq \ul u \leq \ol u$, subsolution and supersolution respectively to \eqref{eq waves}, such that
\begin{equation}\label{supersolution}
\sup_{x\cdot e - ct \geq \delta}\ol u(t,x) \underset{\delta \to + \infty}{\longrightarrow} 0\quad \text{ and }\quad \sup_{x\cdot e - ct \leq \delta}\vert \ol u(t,x) - U(x)\vert\underset{\delta\to - \infty}{\longrightarrow} 0
\end{equation}
and such that there are $\alpha,\beta \in \R$, $\alpha<\beta$ such that
\begin{equation}\label{subsolution}
\inf_{t \in \R}\left(\inf_{\alpha\leq x\cdot e -ct \leq \beta}\ul u(t,x)\right) >0.
\end{equation}
\end{prop}
\begin{proof}
\emph{Step 1. Construction of the supersolution $\ol u$.}\\
Take $c>c^\star(e)$. Owing to Proposition \ref{dec}, we can find $\rho>0$ so that $\lambda_1(\rho,c) \leq 1$. Define
$$
w(t,x) := \phi_{\rho,c}(x)e^{-\rho(x\cdot e -ct)}
$$
and
$$
\ol u(t,x) := \min\{w(t,x), U(x)\}.
$$
It follows from \eqref{L KPP} that
$$
\mc T(w) \leq\mc L(w) = \lambda_1(\rho,c) w \leq  w.
$$
Then
$$
\mc T(\ol u) \leq \min\{ \mc T(w) , \mc T(U) \} \leq \ol u,
$$
hence $\ol u$ is a supersolution of \eqref{eq waves}, and it is readily seen that it satisfies \eqref{supersolution}.

\medskip
\emph{Step 2. Construction of the subsolution $\ul u$.}\\
Take $c>c^{\star}(e)$. Because $\lambda_1(0,c) = \l >1$ and because $\rho \mapsto \lambda_1(\rho,c)$ is continuous, owing to Proposition \ref{dec}, we find $\rho, \rho'$ such that 
$$
0<\rho < \rho' <2\rho
$$
and
$$
\lambda_1(\rho,c) \geq 1 \ \text{ and }\ \lambda_1(\rho',c)<1.
$$
We define
$$
v(t,x) := \phi_{\rho,c}(x)e^{-\rho(x\cdot e -ct)} - M\phi_{\rho',c}(x)e^{-\rho'(x\cdot e -ct)},
$$
where $M$ is large enough so that
$$
v(t,x) \leq 0 \quad \text{ when }\quad x\cdot e -ct \leq 0.
$$
Observe that
$$
v(t,x) \underset{x\cdot e -ct \to +\infty}{\longrightarrow} 0\quad \text{ and } \quad v(t,x) \underset{x\cdot e -ct \to -\infty}{\longrightarrow} -\infty.
$$
For $x \cdot e -ct \geq 0$, we have
$$
v^2(t,x) \leq  \phi_{\rho,c}^2(x) e^{-2\rho(x\cdot e -ct)} \leq \left\|\frac{\phi_{\rho,c}^2}{\phi_{\rho',c}}\right\|_{L^\infty} \phi_{\rho',c}(x)e^{-\rho'(x\cdot e -ct)}.
$$
We define
$$
\ul u := \max\{ v , 0 \}.
$$
Then
$$
\begin{array}{rl}
\mc T(\ul u) &\geq \mc L(\ul u) -C\mc L(\ul u^2)\\
&\geq \mc L(v) - C\mc L(v^2) \\
&\geq \mc L(v)- C\left\|\frac{\phi_{\rho,c}^2}{\phi_{\rho',c}}\right\|_{L^\infty}\mc L(\phi_{\rho',c}(x)e^{-\rho'(x\cdot e -ct)}) \\
&\geq \mc L(v) - C\left\|\frac{\phi_{\rho,c}^2}{\phi_{\rho',c}}\right\|_{L^\infty}\lambda_1(\rho',c)\phi_{\rho',c}(x)e^{-\rho'(x\cdot e -ct)}.
\end{array}
$$
%$$
%\begin{array}{rl}
%T(w) &\geq L(w) - C\left\|\frac{\phi_\lambda^2}{\phi_{\lambda'}}\right\|_{L^\infty}L(\phi_{\lambda'}(x)e^{-\rho'(x\cdot e -ct)}) \\
%&\geq L(w) - C\left\|\frac{\phi_\lambda^2}{\phi_{\lambda'}}\right\|_{L^\infty}L_{c}(\lambda')\phi_{\lambda'}(x)e^{-\rho'(x\cdot e -ct)}
%\end{array}
%$$
Because
$$
\mc L(v) = \lambda_1(\rho,c)\phi_{\rho,c}(x)e^{-\rho(x\cdot e -ct)} - M\lambda_1(\rho',c)\phi_{\rho',c}(x)e^{-\rho'(x\cdot e -ct)},
$$
we finally get
\begin{multline*}
\mc T(\ul u) \geq \lambda_1(\rho,c)\phi_{\rho,c}(x)e^{-\rho(x\cdot e -ct)} \\ - \left( M \lambda_1(\rho',c) + C\left\|\frac{\phi_{\rho,c}^2}{\phi_{\rho',c}}\right\|_{L^\infty}\lambda_1(\rho',c)\right)\phi_{\rho',c}(x)e^{-\rho'(x\cdot e -ct)}.
\end{multline*}
Owing to our choice of $\rho, \rho'$, we can increase $M$ if needed to ensure that
\begin{equation*}
M \lambda_1(\rho',c) + C\left\|\frac{\phi_{\rho,c}^2}{\phi_{\rho',c}}\right\|_{L^\infty}\lambda_1(\rho',c)\leq M,
\end{equation*}
which yields
$$
\mc T(\ul u) \geq v.
$$
Because $\ul u \geq 0$, we have $\mc T(\ul u) \geq 0$, and then
$$
\mc T(\ul u)\geq \ul u,
$$
that is, $\ul u$ is a subsolution of \eqref{eq waves}. By construction, it satisfies \eqref{subsolution}. Moreover, up to increasing $M$, we can ensure that $\ul u \leq \ol u$.
\end{proof}
We will need the following technical lemma in the sequel.
\begin{lemma}\label{lemma T}
For every $\e>0$, there is $\delta>0$ such that, for every $u\in C^{0}(\R^{N+1})$, for every $t_1,t_2 \in \R$ and $x_1, x_2 \in \R^N$ such that $\vert t_1 -t_2 \vert + \vert x_1 -x_2 \vert \leq \delta$, we have
$$
\vert\mc T(u)(t_1,x_1) - \mc T(u)(t_2,x_2) \vert \leq \e.
$$
\end{lemma}
This lemma implies that the image of $C^0(\R^{N+1})$ by $\mc T$ is a set of equicontinuous functions.
%Owing to  hypothesis \eqref{hyp gamma}, it is also readily seen that the functions in $\mc T(C^{0}(\R^N))$ are uniformly bounded.\nota{dire Ascoli?}
\begin{proof} \emph{Step 1. Uniform continuity with respect to $t$.}
Let $u\in C^{0}(\R^{N+1})$ and define $v := \mc T u$. Let us prove that:
\begin{equation}\label{unif c}
\forall \e >0, \ \exists \delta>0 \ \text{ such that } \ \vert t_2 -t_1  \vert \leq \delta \implies \sup_{x\in \R^N}\vert v(t_2,x) - v(t_1,x) \vert \leq \e.
\end{equation}
We argue by contradiction. We assume that there are $\e>0$ and three sequences $(t_1^n)_{n\in\N},(t_2^n)_{n\in\N}, (x_n)_{n\in \N}$, $t_1^n, t_2^n \in \R$, $x_n\in \R^N$, such that, for every $n\in \N$, $t_1^n<t_2^n$, $\vert t_1^n - t_2^n \vert \leq \frac{1}{n}$ and
$$
\vert v(t_2^n,x_n) - v(t_1^n,x_n)\vert \geq \e.
$$
First, up to a change of variable, we have
$$
v(t,x) = \int_{-\infty}^{t}\int_{\R^N} \Gamma(t-\tau,x,y)g(u(\tau,y))dy d\tau.
$$
Hence, for every $t_1<t_2$ and $x\in \R^N$, we have
\begin{multline*}
\vert v(t_2,x) - v(t_1,x)\vert \leq \int_{-\infty}^{t_1}\int_{\R^N} \vert \Gamma(t_2-\tau,x,y) - \Gamma(t_1 -\tau,x,y)\vert g(u(\tau,y))dy d\tau \\ + \int_{t_1}^{t_2}\int_{\R^N} \Gamma(t_2-\tau,x,y)g(u(\tau,y))dy d\tau.
\end{multline*}
Owing to the boundedness of $g$, we find that, up to another change of variable,
\begin{multline*}
\vert v(t_2,x) - v(t_1,x)\vert \leq \|g\|_{L^{\infty}} \int_{0}^{+\infty}\int_{\R^N} \vert \Gamma(t_2-t_1+\tau,x,y) - \Gamma(\tau,x,y)\vert dy d\tau \\ + \|g\|_{L^{\infty}}  \int_{0}^{t_2-t_1}\int_{\R^N} \Gamma(\tau,x,y)dy d\tau.
\end{multline*}
Let us define $\eta_n := t_2^n - t_1^n$. From the above computations, it follows that
\begin{multline*}
\e \leq \|g\|_{L^{\infty}} \int_{0}^{+\infty}\int_{\R^N} \vert \Gamma(\tau+\eta_n,x_n,y) - \Gamma(\tau,x_n,y)\vert dy d\tau \\ + \|g\|_{L^{\infty}}  \int_{0}^{\eta_n}\int_{\R^N} \Gamma(\tau,x_n,y)dy d\tau.
\end{multline*}
Owing to the periodicity hypothesis \eqref{hyp per}, we have
\begin{multline}\label{abs c}
\e \leq \|g\|_{L^{\infty}} \int_{0}^{+\infty}\int_{\R^N} \vert \Gamma(\tau+\eta_n,\hat x_n,y) - \Gamma(\tau,\hat x_n,y)\vert dy d\tau \\ + \|g\|_{L^{\infty}}  \int_{0}^{\eta_n}\int_{\R^N} \Gamma(\tau,\hat x_n,y)dy d\tau,
\end{multline}
where $\hat x_n \in [0,1)^N$ is such that $x_n - \hat x_n \in \Z^N$. By compactness, we assume that $\hat x_n$ converges to some $\hat x~\in~[0,1]^N$ as $n$ goes to $+\infty$. Observe that
\begin{multline*}
 \int_{0}^{+\infty}\int_{\R^N} \vert \Gamma(\tau+\eta_n,\hat x_n,y) - \Gamma(\tau,\hat x_n,y)\vert dy d\tau \leq \\
 2 \int_{0}^{+\infty}\int_{\R^N} \vert \Gamma(\tau,\hat x_n,y) - \Gamma(\tau,\hat x,y)\vert dy d\tau +  \int_{0}^{+\infty}\int_{\R^N} \vert \Gamma(\tau+\eta_n,\hat x,y) - \Gamma(\tau,\hat x,y)\vert dy d\tau.
\end{multline*}
The first term on the right-hand side goes to zero as $n$ goes to $+\infty$, because $\hat x_n$ goes to $\hat x$ and thanks to \eqref{hyp gamma 1}. The second term goes to zero as $n$ goes to $+\infty$ because $\Gamma(\cdot, \hat x, \cdot)$ is in $L^1$.
%(to see this, it is sufficient to approximate $\Gamma(\cdot,\hat x,\cdot)$ by a continuous and compactly supported function, and to use the triangular inequality and the uniform continuity of the approximation).

A similar argument shows that $\int_{0}^{\eta_n}\int_{\R^N} \Gamma(\tau,\hat x_n,y)dy d\tau$ goes to zero as $n$ goes to $+\infty$. This contradicts \eqref{abs c}. Hence \eqref{unif c} holds true.

\medskip
\emph{Step 2. Uniform continuity.}
Take $t_1, t_2 \in \R$ and $x_1, x_2 \in \R^N$. We have
\begin{multline*}
\vert v(t_1,x_1) - v(t_2,x_2) \vert\leq \sup_{x\in \R^N}\vert v(t_1,x) - v(t_2,x) \vert + \vert v(t_2,x_1) - v(t_2,x_2)\vert \\
  \leq \sup_{x\in \R^N}\vert v(t_1,x) - v(t_2,x) \vert + \| g\|_{L^{\infty}} \int_{0}^{+\infty}\int_{\R^N}\vert \Gamma(\tau,x_1,y) -\Gamma(\tau,x_2,y)\vert dy d\tau.
\end{multline*}
Therefore, owing to the first step and to hypothesis \eqref{hyp gamma 1}, the result follows.
\end{proof}
We are now in position to construct traveling waves solutions to \eqref{eq waves}.

%Using the supersolutions and subsolutions provided by Proposition \ref{prop sol}, we now construct waves in every direction $e\in \mathbb{S}^{N-1}$ with speed $c>c^{\star}(e)$.

\begin{prop}
Assume that $\lambda_1>1$ and let $U$ denote the unique positive periodic solution of \eqref{eq stat} given by Proposition \ref{sol stat}. For every direction $e\in \mathbb{S}^{N-1}$ and for every speed $c>c^{\star}(e)$, there is a traveling wave solution to \eqref{eq waves} connecting $0$ to $U$.
\end{prop}

\begin{proof}
\emph{Step 1. Construction of a solution.}\\
Let $e\in\mathbb{S}^{N-1}$ and $c>c^{\star}(e)$. Let $\ul u, \ol u$ be given by Proposition \ref{prop sol}.
We define a sequence of functions $(u_n)_{n\in\N}$ by
$$
u_0 = \ol u\quad \text{ and } \quad u_{n+1} = \mc T u_n \quad \text{ for }\ n\geq 0.
$$
Because $\ol u$ is a supersolution of \eqref{eq waves} and because $\mc T$ is order-preserving, it is readily seen that the sequence of functions $(u_n)_{n\in\N}$ is decreasing. We define its pointwise limit
$$
v(t,x) := \lim_{n\to+\infty} u_n(t,x).
$$
Because $\ul u$ is subsolution of \eqref{eq waves} and because $\ul u \leq \ol u$, we have that $\ul u\leq u_n$ for every $n \geq 1$, and then
\begin{equation}\label{u v}
\ul u \leq v \leq \ol u.
\end{equation}
Owing to Lemma \ref{lemma T}, and using the Ascoli-Arzel\`a theorem, we find that the convergence of $u_n$ to $v$ is locally uniform in $(t,x) \in \R^{N+1}$ as $n$ goes to $+\infty$. This implies that, for $(t,x) \in \R^{N+1}$, $\mc T u_n (t,x) \to \mc T v(t,x)$ as $n$ goes to $+\infty$.
%\begin{Romain}
%
%
%Indeed, take $\e>0$ and $R>0$ large enough so that
%
%$$
%\int_0^{+\infty} \int_{B_R^c} V(x,y)e^{-\mu(y)\tau}  d\tau dy = \int_{B_R^c} \frac{V(x,y)}{\mu(y)} dy \leq \e.
%$$
%Then,
%\begin{multline*}
%\vert T(u_n)(t,x) - T(v)(t,x) \vert \leq \int_0^{+\infty} \int_{\R^N} V(x,y)e^{-\mu(y)\tau} \vert g(u_{n}(t-\tau,y)) - g(v(t-\tau,y))\vert d\tau dy \\
%\leq \int_0^{R} \int_{B_R} V(x,y)e^{-\mu(y)\tau} \vert g(u_{n}(t-\tau,y)) - g(v(t-\tau,y))\vert d\tau dy + 2\int_R^{+\infty} \int_{B_R^c} V(x,y)e^{-\mu(y)\tau}d\tau dy \\
%\leq \sup_{\substack{s \in t-R \\ y \in B_R}}\vert u_n(t,x) - v(t,x) \vert + 2\e,
%\end{multline*}
%and the local uniform convergence of $u_n$ to $v$ allows us to conclude the pointwise convergence of $Tu_n$ to $Tv$ as $n$ goes to $+\infty$. Therefore, $Tv = v$, i.e., $v$ is solution of~\eqref{eq waves}.
%\end{Romain}

\medskip
\emph{Step 2. Proving that $v$ is a wave.}\\
Because of \eqref{u v}, we have that

$$
\sup_{\substack{x\cdot e - ct \geq \delta}} v(t,x)\leq \| \phi_{\rho,c}\|_{L^{\infty}}e^{-\rho \delta} \underset{\delta \to +\infty}{\longrightarrow} 0.
$$
It remains to prove that
$$
\sup_{\substack{ x\cdot e - ct \leq \delta}} \vert v(t,x) - U(x) \vert \underset{\delta \to -\infty}{\longrightarrow} 0.
$$
We consider two sequences $t_n$ and $x_n$ such that
 $$
 x_n\cdot e -c t_n \to -\infty \quad \text{ and } \quad \vert U(x_n) - v(t_n,x_n)\vert \to \lim_{\delta \to -\infty}\left(\sup_{\substack{x\cdot e -ct \leq \delta}} \vert U(x) - v(t,x)  \vert\right).
 $$
 We take $(k_n)_{n\in\N} \in \Z^N$ such that $ x_n - k_n  := z_n \in [0,1)^N$ and we define
 $$
 v_n(t,x) := v(t+t_n,x+k_n).
 $$
Owing to the periodicity hypothesis \eqref{hyp per}, we have $v_n = \mc T v_n$. Thanks to Lemma \ref{lemma T}, we can extract from $v_n$ a sequence that converges locally uniformly to a limit $v_\infty$. Moreover, up to another extraction, we assume that $z_n$ converges to some $z \in [0,1]^N$ as $n$ goes to $+\infty$. Now, because $v\leq U$ by construction and by definition of the sequences $(t_n)_{n\in \N}, (x_n)_{n\in \N}$, we have
$$
\forall t\in \R, \ \forall x\in \R^N,\quad U(x) - v_{\infty}(t,x) \leq U(z) - v_{\infty}(0,z),
$$
hence
$$
\forall t\in \R, \quad v_{\infty}(0,z) \leq v_{\infty}(t,z).
$$
Observe that, by construction, $v$ is time-increasing, and so is $v_\infty$. Therefore
$$
\forall t \leq 0, \quad v_{\infty}(t,z) = v_{\infty}(0,z).
$$
Because $v_{\infty} = \mc T v_{\infty}$, evaluating at $(t_1,z)$ and $(t_2,z)$, with $t_1<t_2\leq 0$, we find that
$$
\int_{\R^N}\int_{0}^{+\infty}  \Gamma(\tau,z,y)\left( g(v_{\infty}(t_2-\tau,y))-g(v_{\infty}(t_1-\tau,y))\right) d\tau dy =0.
$$
Because $v_{\infty}(t,y)$ is increasing with respect to $t\in \R$, we eventually infer that
$$
\forall t \leq 0,\forall x\in \R^N \quad v_{\infty}(t,x) = v_{\infty}(0,x).
$$
Therefore, 
$$
v_{\infty}(0,\cdot) = T v_{\infty}(0,\cdot),
$$
where $T$ is defined in \eqref{def T}. Owing to Proposition \ref{sol stat}, it follows that either $v_{\infty}(0,\cdot) \equiv~0$ or $v_{\infty}(0,\cdot) \equiv U$.
%
%
%
%
%
%
%
%We derive with respect to $t$ both side of this equation to get
%$$
%\partial_{t}v_\infty(y,x) =  \int_{\R^N}\int_{0}^{+\infty}  V(x,y)e^{-\mu(y)\tau} g^{\prime}(v_\infty(t-\tau,y))\partial_{t}v_\infty(t-\tau, y) d\tau dy.
%$$
%Evaluating at $t=0$ and $x=z$, we get
%$$
% \int_{\R^N}\int_{0}^{+\infty}  V(z,y)e^{-\mu(y)\tau} g^{\prime}(u_\infty(-\tau,y))\partial_{t}u_\infty(-\tau, y) d\tau dy = 0.
%$$
%Each term in the integral being positive, we see that
%$$
%\partial_{t}u_\infty(s, x) = 0 \quad \text{ for all }\ s\leq 0, \ x\in \R^N,
%$$
%that is, there is $\t u$ such that $u_{\infty}(s,x) = \t t(x)$ for all $s\leq 0$. 
%The equation $Tu_\infty(0,x) = u_\infty(0,x)$ yields that $\t u$ is a stationary solution of $T$. Owing to Proposition \ref{stat}, if follows that either $\t u = 0$ or $\t u = u_\infty$. 
Let us show that $v_{\infty}(0,\cdot)$ is not identically equal to zero.

Because $v$ is time non-decreasing, we have, for $s\leq t$,
$$
\ul u(s,x)\leq v(s,x) \leq v(t,x), 
$$
hence
$$
\sup_{s\leq t} \ul u(s,x) \leq v(t,x).
$$
It is readily seen from the shape of $\ul u$ (given by Proposition \ref{prop sol}) that there are $\delta \in \R$ and $\eta>0$ such that
$$
\eta \leq \sup_{s\leq t} \ul u(s,x) \quad \text{ for } \quad x\cdot e - ct \leq \delta.
$$
Then, by definition of $v_n$, we have
$$
\eta \leq v_{n}(t,x) \quad \text{ for } \quad x\cdot e - ct + x_n\cdot e - ct_n \leq \delta.
$$
Because $x_n\cdot e -c t_n \to -\infty$ as $n$ goes to $+\infty$, we find that
$$
\forall t\in \R, \forall x\in \R^N,\quad \eta \leq v_{\infty}(t,x),
$$
which implies that $v_\infty \equiv U$. Hence,
$$
\lim_{\delta \to -\infty}\left(\inf_{\substack{x\cdot e -ct \leq \delta}} \vert  v(t,x) - U(x)  \vert \right)= 0,
$$
this concludes the proof.
\end{proof}

\subsection{Non-existence of waves}\label{sec non}

This section is dedicated to proving Theorem~\ref{th non}, that is, the non-existence of traveling waves for \eqref{eq waves} with speed lesser than $c^\star(e)$.

The key idea is to build subsolutions to \eqref{eq waves}, with support contained into a set of the form $\{ x\in \R^N \ : \ x\cdot e \in[ ct, ct + A] \}$, for some $A>0$ and for $c<c^{\star}(e)$. Then, by comparison, any traveling wave would have to move faster than these subsolutions.\\

The strategy we employ to build these subsolutions is inspired by a similar one in the theory of KPP reaction-diffusion equations (see \cite{GMZ} for instance): the idea is to build subsolutions of the form $\phi(x)e^{-\rho(x\cdot e -ct)}$ - just like we did for supersolutions above, but now with $\phi, \rho$ complex. It relies on some arguments from complex analysis and petrubation theory.\\

In order to build subsolutions for our equation, we need to work with a penalization of $\Gamma$. For $\delta \in (0,1)$, we introduce
$$
\Gamma_\delta(t,x,y) := (1-\delta)\max\{ \Gamma(t,x,y)-\delta,0\}.
$$
Owing to hypothesis \eqref{1 hyp waves per}, $\Gamma_\delta$ is compactly supported in the sense that there is $R>0$ such that
\begin{equation}\label{supp compact}
\Gamma_\delta(t,x,y) = 0 \quad \text{ for }\  \vert x - y \vert >R.
\end{equation}
Up to taking $\delta>0$ small enough, the kernel $\Gamma_\delta \leq \Gamma$ satisfies the same hypotheses than $\Gamma$ (we need $\delta>0$ to be small enough in order to have the non-degeneracy condition: $\Gamma_\delta(t,x,y) >\e$ when $t+\vert x-y\vert <r$, for some $r,\e>0$). We define the penalized operator
$$
\mc L^\delta u :=  \int_{\R^N}\int_{0}^{+\infty}\Gamma_\delta(\tau,x,y)u(t-\tau,y) d\tau dy.
$$
We also define the operator $\mc S_{\rho,c}^\delta$ in a similar way, and we let $\lambda_1^\delta(\rho,c)$ denote its principal periodic eigenvalue, and then we let $c^{\star}_\delta(e)$ denote the spreading speed defined by formula \eqref{def c}, with $\lambda^\delta_1(\rho,c,e)$ instead of $\lambda_1(\rho,c,e)$. Using the same arguments as in the proof of Proposition \ref{dec}, we can show that $\lambda_1^\delta(\rho,c,e)$ converges to $\lambda_1(\rho,c,e)$ locally uniformly in $c\geq 0$ when $\delta$ goes to zero. This implies that, for every $e\in \mathbb{S}^{N-1}$, 
$$
c^\star_\delta(e) \underset{\delta \to 0}{\longrightarrow} c^\star(e).
$$
The key result of this section is the existence of subsolution for the penalized linear equation:
\begin{prop}\label{sub sol slow}
Let $\delta \in (0,1)$ and $e\in \mathbb{S}^{N-1}$ be fixed. Then, if $c<c^{\star}_\delta(e)$ is close enough to $c^{\star}_\delta(e)$, there is a function $v(t,x)$ that is continuous in $t,x$, non-negative, non-zero and such that there is $A>0$ such that $v(t,x) =0$ when $\vert x\cdot e -ct\vert >A$, which satisfies
\begin{equation}\label{L sub}
\forall t \in \R, \ \forall x\in \R^N, \quad \mc L^\delta v(t,x) \geq v(t,x).
\end{equation}
\end{prop}
Before presenting the proof of Proposition \ref{sub sol slow}, let us explain how it yields Theorem~\ref{th non}.

\begin{proof}[Proof of Theorem \ref{th non}]
We argue by contradiction. Assume that there is a traveling wave $u$ solution of \eqref{eq waves} in the direction $e\in \mathbb{S}^{N-1}$ with speed $c\in [0,c^{\star}(e))$, connecting $U$ to $0$ for some $U>0$, $U\in C^0_{per}(\R^N)$. Let $\delta>0$ be small enough so that $c^{\star}_\delta(e) > c$.

%\begin{equation*}
%\begin{cases}
%\mathcal{T}^\delta u := (1-\delta)\int_{\R^N}\int_{0}^{+\infty}  \Gamma(\tau,x,y)g(u(t-\tau,y)) d\tau dy, \\
%\mc L^\delta u :=  (1-\delta)\int_{\R^N}\int_{0}^{+\infty}   g^{\prime}(0)\Gamma(\tau,x,y)u(t-\tau,y) d\tau dy.
%\end{cases}
%\end{equation*}

Owing to Proposition \ref{sub sol slow}, we can find $\ol c \in (c,c^{\star}_\delta(e))$ such that there is a function $v$ such that
%, defined by~\eqref{def v c}, that satisfies
$$
\mc L^\delta v \geq v,
$$
and that travels with speed $\ol c$ in the direction $e$. Let $\e>0$ be such that, up to a translation in time, we have,
$$
\forall t\leq 0, \quad \e v(t,\cdot) \leq u(t,\cdot).
$$
Moreover, up to decreasing $\e>0$, we ensure that
$$
1-\delta \leq 1-C\e\|v \|_{L^{\infty}},
$$
where $C$ is the constant in \eqref{hyp g}.
We define
$$
T^\star := \max \{ T \geq 0 \ : \e v(t,\cdot) \leq u(t,\cdot) \ \text{ for all }\ t\in [0,T] \}.
$$
It is readily seen that, for $t\leq T^\star$ and for $x\in \R^N$,
\begin{multline}\label{non existence abs}
\e v(t,x) \leq \mc L^\delta (\e v)(t,x) \leq \frac{(1-\delta)}{(1-C\e \| v\|_{L^{\infty}})}\mc T(\e v)(t,x)\\ \leq \mc T(\e v)(t,x) \leq \mc T(u)(t,x) = u(t,x). 
\end{multline}
Now, assume that there is $x^\star$ such that
$$
\e v(T^\star,x^\star) = u(T^\star,x^\star).
$$
Then, evaluating \eqref{non existence abs} at $t = T^\star$ and $x= x^\star$, we would find that $\mc T(\e v)(T^\star,x^\star) \leq \mc T(u)(T^\star,x^\star)$, which would imply that $\e v(t,x) = u(t,x)$ for every $t\leq T^\star, x\in \R^N$, which is not possible, because $ v(t,x)$ has support contained in a set $\vert x \cdot e - \ol c t \vert < A$, for some $A>0$.

If there is not such $x^\star$, we can find a sequence $(x_n)_{n\in\N}$ such that
$$
u(T^\star,x_n) - \e v(T^\star,x_n) \underset{n\to+\infty}{\longrightarrow} 0,
$$
and arguing as in the proof of Proposition \ref{sol stat}, second case, we would again reach a contradiction. 
%Now, by definition of $T^{\star}$ and by continuity, we can find a sequence $(x_{n})_{\n in\N}$ such that
%$$
%u(T^\star,x_n) - \e v(T^\star,x_n) \underset{n\to+\infty}{\longrightarrow} 0.
%$$
%We define the sequences of translated functions
%$$
%u_n(t,x) := u(t,
%$$
\end{proof}

Before turning to the proof of Proposition \ref{sub sol slow}, we start with two intermediary results. In order to simplify the notations, we assume that $\delta$ is fixed and we omit it in the computations, that is, we drop the index $\delta$ everywhere.

We just have to keep in mind that $\Gamma_\delta$ is compactly supported in the sense of \eqref{supp compact}. This is actually only used in the proof of Proposition \ref{sub sol slow}, and not in the following intermediary results, that are true for general kernels $\Gamma$.

We start with a technical result. For $z\in \C$, we let $\Re(z), \Im(z)$ denote the real and imaginary parts of $z$.
\begin{prop}\label{prop comp}
For $\rho,c\geq 0$, let $\lambda_1(\rho,c)$ and $\phi_{\rho,c}$ denote the principal periodic eigenvalue and eigenfunction of $\mc S_{\rho,c}$, normalized so that $\phi_{\rho,c}(0)=1$. For every $c\geq 0$, the maps
$$
\rho \mapsto \lambda_1(\rho,c) \in \mathbb{C}
$$
and
$$
\rho \mapsto \phi_{\rho,c} \in C_{per}^{0}(\R^N, \mathbb{C})
$$
can be holomorphically extended to a complex neighborhood of the positive real axis $\{ \rho \geq 0\}$. The complex-valued functions still satisfy $\mc S_{\rho,c}\phi_{\rho,c} = \lambda_1(\rho,c)\phi_{\rho,c}$. In addition, $\rho \mapsto \Re(\phi_{\rho,c})$ and $\rho \mapsto \Im(\phi_{\rho,c})$ are continuous (with respect to the $C^{0}(\R^N)$ topology).
\end{prop}
Since the operator $\mc S_{\rho,c}$ is compact and holomorphic with respect to $\rho$, and since the principal eigenvalue $\lambda_1(\rho,c)$ is isolated and has multiplicity equal to $1$ (owing to the Krein-Rutman Theorem), the proposition above follows from standard perturbation theory, we refer to \cite[Chapter~7]{Kato} for the details.

\begin{lemma}\label{lemma inf}
We have
$$
\lambda_1(\rho,c) \underset{\rho \to +\infty}{\longrightarrow}+\infty,
$$
and this limit holds locally uniformly in $c\geq 0$.
%Also, we have, for any $e\in \mathbb{S}^{N-1}$,
%$$
%c^{\star}_\delta(e) \underset{\delta \to 0}{\longrightarrow} c^\star(e).
%$$
\end{lemma}
\begin{proof}
Take $\ol x$  such that $\phi_{\rho,c}(\ol x) = \min_{\R^N}\phi_{\rho,c}$, where $\phi_{\rho,c}$ is a positive principal periodic eigenfunction associated with $\lambda(\rho,c)$. Then
\begin{multline*}
\lambda_1(\rho,c)  \geq \int_{\tau >0}\int_{y\in \R^N} \Gamma(\tau,\ol x, y)e^{-\rho c \tau}e^{-\rho (y-\ol x)\cdot e}dy d\tau \\ = \int_{\tau >0}\int_{z\in \R^N} \Gamma(\tau,\ol x, \ol x + z)e^{-\rho c \tau}e^{-\rho z\cdot e}dz d\tau.
\end{multline*}
Now, by hypothesis on $\Gamma$ (see Section \ref{sec hyp}), we have that there are $r,\epsilon$ so that $\Gamma>\e$ if $t + \vert x - y\vert < r$. Hence
$$
\lambda_1(\rho,c)  \geq  \int_{\tau \in(0,\frac{r}{2})}\int_{\vert z\vert < \frac{r}{2}} \Gamma(\tau,\ol x, \ol x + z)e^{-\rho c \tau}e^{-\rho z\cdot e}dz d\tau 
\geq \e \frac{1-e^{-\frac{\rho cr}{2}}}{\rho c} \int_{\vert z\vert \leq \frac{r}{2}}e^{-\rho z \cdot e}dz d\tau.
$$
The rightmost term in these inequalities goes to $+\infty$ as $\rho$ goes to $+\infty$ locally uniformly in $c\geq 0$.
\end{proof}

We are now in position to prove Proposition \ref{sub sol slow}. Let us explain how we build the subsolutions $v$.

First, observe that, owing to Lemma \ref{lemma inf}, to the definition \eqref{def c} of $c^{\star}(e)$ and to the continuity of $\lambda_1(\rho,c)$, we have that there is $\rho^\star$ such that
$$
\l(\rho^\star,c^{\star}(e)) =1.
$$
Owing to Proposition \ref{prop comp} and to Rouch\'e's theorem, there is $\eta>0$ such that, for every $c\in (c^{\star}(e)-\eta,c^{\star}(e)]$, there is $\rho(c) \in \C$ that satisfies $\lambda_1(\rho(c),c)=1$. Moreover, 
%$c \in (c^{\star}(e)-\eta,c^{\star}(e)] \mapsto\rho(c)$ is continuous and 
$\rho(c) \to \rho^\star$ as $c$ goes to $c^{\star}(e)$. We define
\begin{equation}\label{def v c}
v_c(t,x) := \Re(\phi_{\rho(c),c}(x)e^{-\rho(c)(x\cdot e -ct)}).
\end{equation}
Therefore, we have $\mc L v_c = v_c$.
We let $\phi_R, \phi_I$ denote the real and imaginary parts of $\phi_{\rho(c),c}$ (these are real continuous periodic functions) and $\rho_R, \rho_I$ the real and imaginary parts of $\rho(c)$. As $c$ goes to $c^{\star}(e)$, we have $\phi_{R} \to \phi_{\rho^\star,c^{\star}(e)}$, $\phi_I \to 0$, uniformly in $x$, and $\rho_R \to \rho^\star$ and $\rho_I \to 0$. We take $c$ close enough to $c^{\star}(e)$ such that
$$
\min_{[0,1]^N}\phi_{R} >\max_{[0,1]^N}\vert \phi_{I} \vert.
$$
We can rewrite \eqref{def v c} as follows:
\begin{equation}\label{vri}
v_c(t,x) = \Big( \phi_R(x)\cos(\rho_I(x\cdot e -ct)) + \phi_I(x)\sin(\rho_I(x\cdot e -ct)) \Big)e^{-\rho_R(x\cdot e -ct)}.
\end{equation}
%and $\rho_I <\e$.\nota{quoi ? Dire que c'est �gal?}
%%%%%%%%%%%%%%%%%%%%%%%%%%%%%%%%%%%%%%%%%%%%%%%%%%%%%%%%%%
%%%%%%%%%%%%%%%%%%%%%%%%%%%%%%%%%%%%%%%%%%%%%%%%%%%%%%%%%%
%%%%%%%%%%%%%%%%%%%%%%%%%%%%%%%%%%%%%%%%%%%%%%%%%%%%%%%%%%
Observe that, if $x\cdot e -ct = \pm\frac{3\pi}{4\rho_I}$, then
$$
v_c(t,x) = \frac{\sqrt{2}}{2}(-\phi_R(x)\pm \phi_I(x))e^{\mp\frac{3\pi \rho_R}{4\rho_I}}<0.
$$
%Let us introduce a notation: if $z \in \R$, we set $[z]^+ := \max\{ z, 0\}$.
We define
\begin{equation}\label{v plus}
v_c^+(t,x) = 
\begin{cases}
\max\{ v_c(t,x) , 0\} \quad &\text{if} \quad  \vert x\cdot e -ct \vert \leq \frac{3\pi}{4\vert \rho_I\vert}, \\
0 \quad &\text{elsewhere}.
  \end{cases}
\end{equation}
Then, $v_c^+$ is continuous, non-negative, not everywhere equal to zero. Let us show that it is the good candidate to prove Proposition \ref{sub sol slow}, that is, let us show that
$$
\mc L v_c^+ \geq v_c^+.
$$
Remember that we dropped the index $\delta$, and that $\Gamma$ (that is, $\Gamma_\delta$) is compactly supported in the sense of \eqref{supp compact}. It is used in the following proof.

\begin{proof}[Proof of Proposition \ref{sub sol slow}]
For notational simplicity, we define $\e := \vert \rho_I\vert$. Up to taking $c$ closer to $c^{\star}(e)$, we can make $\e$ as small as needed. Clearly, \eqref{L sub} holds true for $(t,x)$ such that $v_c(t,x)\leq 0$. If $(t,x)$ is such that  $\vert x\cdot e -ct \vert \leq \frac{3\pi}{4 \e}$ and such that $v_c(t,x)\geq 0$, then $v_c^+(t,x)=v_c(t,x) = \mc L v_c(t,x)$. Therefore, to prove that \eqref{L sub} holds true, it is sufficient to show that, for $(t,x)$ such that $\vert x\cdot e -ct\vert \leq \frac{3\pi}{4\e}$,
\begin{equation*}
\int_{\R^N}\int_{0}^{+\infty}\Gamma(\tau,x,y) v_c^+(t-\tau,y)d\tau dy \geq \int_{\R^N}\int_{0}^{+\infty}\Gamma(\tau,x,y) v_c(t-\tau,y)d\tau dy.
\end{equation*}
To compare those two integrals, we break them into three parts. We define
$$
\begin{cases}
I_1 := \int_{y\cdot e \in [ct -\frac{\pi}{\e} , ct +\frac{3\pi}{4\e}]}\int_{0}^{+\infty}\Gamma(\tau,x,y) (v_{c}^+(t-\tau,y) - v_c(t-\tau,y))d\tau dy, \\
I_2 := \int_{y\cdot e \in [ct +\frac{3\pi}{4\e} , ct +\frac{\pi}{\e}]}\int_{0}^{+\infty}\Gamma(\tau,x,y) (v_{c}^+(t-\tau,y) - v_c(t-\tau,y))d\tau dy,\\
I_3 := \int_{\vert y\cdot e -ct \vert \geq \frac{\pi}{\e}}\int_{0}^{+\infty}\Gamma(\tau,x,y) (v_{c}^+(t-\tau,y) - v_c(t-\tau,y))d\tau dy.
\end{cases}
$$
Let us prove that we can take $\e>0$ small enough so that, for every $t\in \R, x \in \R^N$, $I_1+I_2+I_3 \geq 0$. First, because we assume that $\vert x\cdot e - ct\vert \leq \frac{3\pi}{4\e}$, we see that, if $y$ is such that $\vert y\cdot e -ct \vert \geq \frac{\pi}{\e}$, then $\vert x - y\vert \geq \frac{\pi}{4\e}$. Therefore, up to taking $\e$ small enough, we have $\Gamma(\tau, x,y) =0$ (thanks to the hypothesis that $\Gamma$ is compactly supported), then $I_3 = 0$, for every $t \in \R$, and every $x$ such that $\vert x\cdot e - ct\vert \leq \frac{3\pi}{4\e}$.

\medskip
\emph{Step 1. Estimate on $I_1$.}\\
Let $y\in \R^N$ be such that $y\cdot e \in [ct -\frac{\pi}{\e} , ct +\frac{3\pi}{4\e}]$. We define $\tau_1, \tau_2$ such that
$$
y\cdot e  = c(t-\tau_1)+\frac{3\pi}{4\e}, \quad y\cdot e = c(t-\tau_2)+\frac{5\pi}{4\e}.
$$
We now estimate $v_{c}^+(t-\tau,y), v_c(t-\tau,y)$ for $\tau$ in $[0,\tau_1], [\tau_1,\tau_2]$ and $[\tau_2,+\infty]$. By definition of $v_{c}^+$, we have
$$
v_{c}^+(t-\tau,y) \geq v_c(t-\tau,y) \quad \text{ for }\ \tau \in [0,\tau_1].
$$
Indeed, for such $\tau$, we have $y\cdot e -c(t-\tau) \in [-\frac{\pi}{\e}, \frac{3\pi}{4\e}]$. When $y\cdot e - c(t-\tau) \in [-\frac{3\pi}{4\e}, \frac{3\pi}{4\e}] $, the definition of $v_c^+$ implies that $v_c^+ \geq v_c$. If $y\cdot e - c(t-\tau) \in [-\frac{\pi}{\e}, -\frac{3\pi}{4\e}]$, we have $v_c(t-\tau,y) \leq 0 = v_c^+(t-\tau,y)$.\\

Define $\kappa := \frac{\sqrt{2}}{2}(\min \phi_R - \max \vert \phi_I \vert) >0$. It follows from \eqref{vri} that
%Up to an easy computation, we can infer from \eqref{vri} that, for $\kappa := \frac{\sqrt{2}}{2}(\min \phi_R - \max \phi_I) >0$, we have
$$
v_c(t-\tau,y) \leq -\kappa e^{-\rho_R(y\cdot e -c(t-\tau))}\quad \text{ for }\ \tau \in [\tau_1,\tau_2].
$$
Finally, it is readily seen from \eqref{vri} that, for $M := \max{\phi_R} + \max \vert \phi_I \vert$,
$$
v_c(t-\tau,y) \leq M e^{-\rho_R(y\cdot e -c(t-\tau))}\quad \text{ for }\ \tau \in [\tau_2,+\infty].
$$
Now, we compute
\begin{multline*}
\int_{0}^{+\infty}\Gamma(\tau,x,y)(v_c^+(t-\tau,y)-v_c(t-\tau,y))d\tau  \geq -\int_{\tau_1}^{+\infty}\Gamma(\tau,x,y)v_c(t-\tau,y)d\tau \\
\geq \kappa \int_{\tau_1}^{\tau_2} \Gamma(\tau,x,y)e^{-\rho_R(y\cdot e -c(t-\tau))}d\tau - M \int_{\tau_2}^{+\infty} \Gamma(\tau,x,y)e^{-\rho_R(y\cdot e -c(t-\tau))}d\tau.
\end{multline*}
Let us show that this is non-negative, up to taking $\e$ small enough (independant of $x,t,y$). To do this, define $T_\e = \tau_2 - \tau_1 = \frac{\pi}{2c \e}$ and
$$
h(s) = \kappa \int_{s}^{s+ T_\e} \Gamma(\tau,x,y)e^{-\rho_R(y\cdot e -c(t-\tau))}d\tau - M \int_{s + T_\e}^{+\infty} \Gamma(\tau,x,y)e^{-\rho_R(y\cdot e -c(t-\tau))}d\tau.
$$
We have
$$
h^\prime(s) = e^{-\rho_R (y\cdot e - c(t-s))} \left( \kappa\left(\Gamma(s+T_\e,x,y)e^{-\rho_R c T_\e} -\Gamma(s,x,y)  \right)  + M\Gamma(s+T_\e,x,y)e^{-\rho_R c T_\e} \right).
$$
Therefore, owing to hypothesis \eqref{non exp}, up to taking $\e$ small enough, independent of $t,x,y$, we can ensure that $h^\prime(s)\leq 0$, and because $h(s)\to 0$ as $s$ goes to $+\infty$, then $h(\tau_1)\geq 0$, hence $\int_{0}^{+\infty}\Gamma(\tau,x,y)(v_c^+(t-\tau,y)-v_c(t-\tau,y))d\tau  \geq 0$, for every $t\in \R, x, y\in \R^N$. Therefore, $I_1 \geq 0$.

%Therefore, we find that (we recall that $\tau_1, \tau_2$ depend on $y$)
%\begin{multline*}
%\int_{y\cdot e \in [ct -\frac{\pi}{\e} , ct +\frac{3\pi}{4\e}]}�\int_{0}^{+\infty}\Gamma(\tau,x,y)(v_+(t-\tau,y)-v(t-\tau,y))d\tau dy\geq \\
%\int_{y\cdot e \in [ct -\frac{\pi}{\e} , ct +\frac{3\pi}{4\e}]}�\left( \left( \kappa(e^{-(\mu+\lambda_R c)\tau_1}-e^{-(\mu+\lambda_R c)\tau_2}) - K e^{-(\mu+\lambda_R c)\tau_2}\right) \frac{V(x,y)}{\mu(y)+\lambda_R c}e^{-\lambda_R(y\cdot e -ct)}\right).
%\end{multline*}

\medskip
\emph{Step 2. Estimate for $I_2$.}\\
Consider now the situation when $y\cdot e - ct \in [\frac{3\pi}{4\e},\frac{\pi}{\e}]$. Then, in this case, $v_c^+(t-\tau,y) =0$ for all $\tau\geq 0$, and
$$
v_c(t-\tau,y) \leq -\kappa e^{-\rho_R(y\cdot e -c(t-\tau))}\quad \text{ for }\ \tau \in [0,\tau_2],
$$
where $\tau_2$ is defined as in the previous step. Therefore, we still have in this case that
\begin{multline*}
\int_{0}^{+\infty}\Gamma(\tau,x,y)(v_c^+(t-\tau,y)-v_c(t-\tau,y))d\tau \\
\geq \kappa \int_{\tau_1}^{\tau_2} \Gamma(\tau,x,y)e^{-\rho_R(y\cdot e -c(t-\tau))}d\tau - M \int_{\tau_2}^{+\infty} \Gamma(\tau,x,y)e^{-\rho_R(y\cdot e -c(t-\tau))}d\tau,
\end{multline*}
and we can conclude as in the previous step: up to taking $\e$ small enough, for every $t,x$ such that $\vert x\cdot e -ct \vert \leq \frac{3\pi}{4\e}$, we have
$$
I_2 \geq  0.
$$
This concludes this step and the proof.
\end{proof}

\section{Application to the SIR model \eqref{SIR 2}}\label{sec comp}

We now apply our results on the renewal equations to the SIR model \eqref{SIR 2}. In the whole section, $\mu,\alpha, K$ satisfy the hypotheses of Theorem \ref{th sir 1}, that is $\mu, \alpha \in C^0_{per}(\R^N)$, $\alpha,\mu >0$ and $K \in C^0(\R^N\times \R^N)$ is periodic, non-negative, non-zero, symmetric and decays faster than any exponential.\\

Let us first explain how the SIR model \eqref{SIR 2} rewrites as the integral equations \eqref{eq}. This fact was observed by several authors in different contexts (see \cite{formulation,Di1} for instance) - we prove it in our setting for the sake of completeness.\\

\begin{prop}\label{equiv SIR int}
Let $S_0\in C^0_{per}(\R^N)$, $I_0 \in C^0(\R^N)$, with $S_0>0$ and $I_0$ non-negative and bounded.

%Let $A$ be a matrix field of class $C_{per}^2(\R^N)$ and let $\mu, \alpha \in C_{per}^1(\R^N)$.

\begin{itemize}

\item If $(S(t,x),I(t,x))$ solves \eqref{SIR 2} with initial datum $(S_0,I_0)$, then $u(t,x) := - \ln \left(\frac{S(t,x)}{S_0(x)}\right)$ solves \eqref{eq} with $\Gamma,f,g$ given by \eqref{gamma 2}.

\item If $u$ solves \eqref{eq} with $\Gamma,f,g$ given by \eqref{gamma 2}, then $S(t,x) := S_0(x)e^{-u(t,x)}$ and $I(t,x) := -\int_{0}^t e^{- \mu(x)(t-\tau)}\partial_{t}S(\tau,x)  d\tau +I_0(x)e^{-\mu(x)t}$ 
%$I(t,x) := \frac{1}{\alpha(x)}\partial_t u(t,x) $
%$I(t,x) := -\int_{0}^t \int_{z\in \R^N}\mathcal{H}(\tau,x,z)\partial_{t}S(t-\tau,z) dz d\tau + \int_{z\in \R^N}\mathcal{H}(t,x,z)I_0(z)dz$ 
solve \eqref{SIR 2} with initial datum $(S_0,I_0)$.

\end{itemize}

\end{prop}
Before turning to the proof of this proposition, observe that, combining it with Proposition \ref{prop cv} gives the existence and uniqueness of solutions of the SIR system~\eqref{SIR 2}.
\begin{proof}[Proof of Proposition \ref{equiv SIR int}.]
Let us show the first point. Observe that, if $(S,I)$ is solution to \eqref{SIR 2}, then the function $u(t,x) := - \ln(\frac{S(t,x)}{S_0(x)})$ is continuous and positive on $(0,+\infty)\times \R^N$.
The second equation of \eqref{SIR 2} rewrites
$$
\partial_t I(t,x) +\mu(x)I(t,x) = -  \partial_t S(t,x).
$$
Therefore,
$$
I(t,x) = -\int_{0}^t e^{- \mu(x)\tau}\partial_{t}S(t-\tau,x)  d\tau +I_0(x)e^{-\mu(x)t}. 
$$
Plotting this in the first equation of \eqref{SIR 2} yields
\begin{multline}\label{a tech}
\partial_tu(t,x) =-\frac{\partial_{t}S(t,x)}{S(t,x)} =   -\alpha(x)\int_{0}^t \int_{y\in \R^N} K(x,y)e^{-\mu(y)\tau}\partial_{t}S(t-\tau,y)  dy d\tau \\ +\alpha(x)\int_{y\in \R^N}K(x,y)I_0(y) e^{-\mu(y)t}dy.
\end{multline}
By definition of $u$, integrating with respect to the $t$ variable between $0$ and $T$ and changing the order the integrals yields
\begin{multline*}
u(T,x) =   -\alpha(x)\int_{0}^T \int_{y\in \R^N} K(x,y)e^{-\mu(y)\tau}(S(T-\tau,y)-S_0(y))  dy d\tau \\ + \alpha(x)\int_0^T\int_{y\in \R^N}K(x,y)I_0(y) e^{-\mu(y)t} dy dt.
\end{multline*}
Now, remembering that $g(z) := 1 -e^{-z}$, we have
\begin{multline*}
u(T,x) =   \alpha(x)\int_{0}^T \int_{y\in \R^N} S_0(y)K(x,y)e^{-\mu(y)\tau}g(u(T-\tau,y))  dy d\tau \\ +\alpha(x)\int_0^T\int_{y\in \R^N}K(x,y)I_0(y) e^{-\mu(y)t}dt dy.
\end{multline*}
which proves the first point. \\

The second point can be proved using the same computations. One has to observe that, if $u$ is the unique solution of \eqref{eq} - provided by Proposition \ref{prop cv} - with $\Gamma,f,g$ given by \eqref{gamma 2}, then, the functions $(S,I)$ defined as fonctions of $u$ in the second point have the required regularity to be considered solutions of \eqref{SIR 2}.
\end{proof}

We now turn to the proofs of Theorems \ref{th sir 1}, \ref{th sir 2}. The idea is to use the change of functions presented in Proposition \ref{equiv SIR int} and to apply Theorems \ref{th threshold}, \ref{th waves} and \ref{th non} to the renewal equation thus obtained.\\

First, observe that, if $\Gamma,f,g$ are given by \eqref{gamma 2} (with $\alpha,\mu, K, S_0$ satisying the above hypotheses),  they satisfy the hypotheses required to apply Theorems \ref{th threshold}, \ref{th waves} and \ref{th non} (these hypotheses are given in Section \ref{sec results}).

\begin{proof}[Proof of Theorem \ref{th sir 1}]
Let $S_0,I_0$ be such that $S_0$ is continuous periodic and strictly positive, and $I_0$ is continuous, non-negative, non-zero, and compactly supported. Let $(S,I)$ the solution of \eqref{SIR 2} arising from this initial datum. \\

Owing to Proposition \ref{equiv SIR int}, we have that $u(t,x) = -\ln(\frac{S(t,x)}{S_0(x)})$ is the solution of \eqref{eq} with $\Gamma,f,g$ given by \eqref{gamma 2}.\\

Let $\l$ be the principal periodic eigenvalue of the operator
 $$
 \phi \mapsto \int_{\R^N}\frac{\alpha(x)S_0(y)}{\mu(y)}K(x,y) \phi(y)dy.
 $$
Let us prove that the epidemic propagates in the sense of Definition \ref{def prop SIR} when $\lambda_1>1$. In this case, Theorem \ref{th threshold} tells us that the epidemic propagates for \eqref{eq} (in the sense of Definition \ref{def prop}). Hence, $u(t,x)$ converges to some $u_{\infty}(x)$, that satisfies, for some $\e,R>0$,
$$
u_{\infty}(x) > \e ,\quad \text{ for } \vert x \vert\geq R.
$$
Because $S(t,x) = S_0(x)e^{-u(t,x)}$, we see that $S(t,x)$ converges to $S_{\infty}(x) := S_0(x)e^{-u_{\infty}(x)}$, and
$$
S_{\infty}(x) < S_0(x)e^{-\e},\quad \text{ for } \vert x \vert \geq R,
$$
hence
$$
\sup_{\vert x \vert >R} (S_{\infty}(x) - S_0(x))<0.
$$
Moreover, because $S(t,x)$ is strictly decreasing with respect to $t$ for every $x$, the result follows by continuity of $S_{\infty}$ and $S_0$.\\

To conclude, let $\mc S := S_0 e^{-U}$, where $U$ is given by Theorem \ref{th threshold}. Then, we have
$$
\vert S_\infty(x) - \mc S(x)\vert \leq \| S_0\|_{L^\infty} \vert e^{-u_\infty(x)} - e^{-U(x)}\vert \leq \| S_0\|_{L^\infty} \vert u_\infty(x) -U(x)\vert,
$$
and Theorem \ref{th threshold} allows to conclude.

The proof for the fading out when $\l\leq 0$ follows the same lines.
\end{proof}

We now turn to the existence and non-existence of waves.
\begin{proof}[Proof of Theorem \ref{th sir 2}]
Let $\alpha,\mu,K$ and $S_{-\infty}$ satisfy the hypotheses of Theorem \ref{th sir 2}. \\

\medskip
\emph{Step 1. Existence of waves.}\\
Assume that $\lambda_1>1$.  Let $c>c^\star(e)$. Then, Theorem \ref{th waves} tells us that the renewal equation \eqref{eq waves} with $\Gamma,g$ given by \eqref{gamma 2} (with $S_0$ replaced by $S_{-\infty}$) admits a traveling waves connecting $0$ to some $U$ with speed $c$ in the direction $e$. Let $u$ be such a wave.\\

Let us define
$$
S(t,x) = S_{-\infty}e^{-u(t,x)}, \quad I(t,x):= -\int_{-\infty}^t e^{- \mu(x)(t-\tau)}\partial_{t}S(\tau,x)  d\tau .
$$
Up to doing the same computations as those done in the proof of Proposition \ref{equiv SIR int}, it is readily seen that $S,I$ solve \eqref{SIR 2}. Let us prove that these are traveling fronts in the sense of Definition \ref{def tw SIR}.

\medskip
\emph{Limit when  $x \cdot e - ct \to -\infty$.}\\
We have
$$
\sup_{x\cdot e -c t \leq -\delta} \vert u(t,x) - U(x) \vert\underset{\delta \to +\infty}{\longrightarrow} 0.
$$
Hence
\begin{equation*}
\begin{array}{rl}
\sup_{x\cdot e -c t \leq -\delta} \vert S(t,x)- S_{-\infty}(x)e^{-U(x)} \vert &= \sup_{x\cdot e -c t \leq -\delta} \vert S_{-\infty}(x)( e^{-u(t,x)}- e^{-U(x)}) \vert \\
&\leq \|S_{-\infty}\|_{L^\infty}\sup_{x\cdot e -c t \leq -\delta} \vert u(t,x)-U(x) \vert \underset{\delta \to +\infty}{\longrightarrow} 0.
\end{array}
\end{equation*}
Let us show that
\begin{equation}\label{I wave 1}
\sup_{x\cdot e -c t \leq -\delta}  I(t,x) \underset{\delta \to +\infty}{\longrightarrow} 0.
\end{equation}
We have
%When $\Gamma,f,g$ are given by \eqref{gamma 2}, we have that
%$$
%\partial_t S (t,x) = -\alpha(x) \int_{y\in \R^N}K(x,y)I(t,y)dy
%$$
$$
I (t,x) = -\int_{-\infty}^{t} \partial_t S(\tau,x)  e^{\mu(x)(\tau-t)}d\tau,
$$
and then
$$
I(t,x) = \int_{-\infty}^{t}S(\tau,x)  \mu(x)e^{\mu(x)(\tau-t)}d\tau - S(t,x).
$$
Let $(x_n)_{n\in \N}\in (\R^N)^\N$ and $(t_n)_{n\in\N} \in \R^\N$ be such that $x_n \cdot e - c t_n := -\delta_n \to -\infty$ as $n$ goes to $+\infty$. Then
%\begin{equation*}
%\begin{array}{rl}
%\begin{multline*}
\begin{align*}
I(t_n,x_n)  &= \int_{-\infty}^{t_n}S(\tau,x_n)  \mu(x_n)e^{\mu(x_n)(\tau-t_n)}d\tau - S_{+\infty}(x_n) +  S_{+\infty}(x_n) - S(t_n,x_n)
\\ &= \int_{-\infty}^{t_n}(S(\tau,x_n) - S_{+\infty}(x_n)) \mu(x_n)e^{\mu(x_n)(\tau-t_n)}d\tau +  S_{+\infty}(x_n) - S(t_n,x_n) \\
&\leq\| S_{-\infty}\|_{L^{\infty}} \|\mu\|_{L^{\infty}}\left(\int_{-\infty}^{t_n - \frac{\delta_n}{2 c}}\mu(x_n)e^{\mu(x_n)(\tau-t_n)}d\tau \right) \hspace{30mm} \\ &\hspace{15mm}+ \left(\int_{t_n - \frac{\delta_n}{2 c}}^{t_n} \mu(x_n)e^{\mu(x_n)(\tau-t_n)}d\tau\right)\sup_{x\cdot e - ct \leq -\frac{\delta_n}{2}}\vert S(t,x) - S_{+\infty}(x)  \vert \\ &\hspace{15mm}+ \vert S_{+\infty}(x_n) - S(t_n,x_n)\vert \\
&\leq \|S_{-\infty}\|_{L^{\infty}}\|\mu\|_{L^{\infty}} e^{-\mu(x_n) \frac{\delta_n}{2c}} + \sup_{x\cdot e - ct \leq -\frac{\delta_n}{2}}\vert S(t,x) - S_{+\infty}(x)  \vert+  \vert S_{+\infty}(x_n) - S(t_n,x_n)\vert.
\end{align*}
%\end{multline*}
%\end{array}
%\end{equation*}
This goes to zero as $n$ goes to $+\infty$, hence \eqref{I wave 1} follows.

\medskip
\emph{Limit when $x \cdot e - ct \to +\infty$.}\\
We have
$$
\sup_{x\cdot e -c t \geq \delta} \vert u(t,x)  \vert \underset{\delta \to +\infty}{\longrightarrow} 0.
$$
Hence
\begin{equation*}
\begin{array}{rl}
\sup_{x\cdot e -c t \geq \delta} \vert S(t,x) - S_{-\infty}(x)\vert &= \sup_{x\cdot e -c t \geq \delta} \vert S_{-\infty}(t,x)(1-e^{-u(t,x)})  \vert \\  
&\leq \|S_{-\infty}\|_{L^\infty} \sup_{x\cdot e -c t \geq \delta} \vert u(t,x) \vert  \underset{\delta \to +\infty}{\longrightarrow} 0.
\end{array}
\end{equation*}
To prove that 
$$
\sup_{x\cdot e -c t \geq \delta} \vert I(t,x) \vert  \underset{\delta \to +\infty}{\longrightarrow} 0,
$$
we could argue as in the first step, however, there is here a simpler argument. Observe that, for $t>0$ and $x\in\R^N$,
$$
\partial_t S(t,x) +\partial_t I(t,x) \leq  0,
$$
hence
$$
0 \leq \sup_{x\cdot e -c t \geq \delta}  I(t,x)  \leq  \sup_{x\cdot e -c t \geq \delta}  \vert S_{-\infty}(x) - S(t,x) \vert,
$$
and the result follows.\\

\medskip
\emph{Step 2. Non-existence of waves.}\\
Assume by contradiction that $c<c^\star(e)$ and that there is a traveling wave solution to \eqref{SIR 2} with this speed in direction $e$. Let $S(t,x)$ be the wave of susceptibles. Then, doing computations similar to those in the proof of Proposition \ref{equiv SIR int} would yield that $u(t,x) = - \ln(\frac{S(t,x)}{S_{-\infty}(x)})$ is a traveling wave solution of \eqref{eq waves} with speed $c$ in the direction $e$. This is impossible owing to Theorem \ref{th non}, hence the result.
\end{proof}

\end{document}